\documentclass[9pt]{amsart}

\usepackage{amsmath,amsfonts,amsthm}
\usepackage{color}
\usepackage{verbatim}
\usepackage{eepic}
\usepackage{epic}
\usepackage{epsfig,subfigure,epstopdf}
\usepackage[]{graphics}

\newtheorem{theorem}{Theorem}[section]
\newtheorem{lemma}{Lemma}[section]

\newtheorem{remark}{Remark}[section]

\numberwithin{equation}{section}

\begin{document}

\newcommand{\al}{\alpha}
\newcommand{\fy}{\varphi}
\newcommand{\la}{\lambda}
\newcommand{\ep}{\epsilon}

\newcommand{\vth}{\vartheta}
\newcommand{\vtht}{{\widetilde \vartheta}}
\newcommand{\rlh}{{\widetilde \varrho}}

\def\tribar{\vert\thickspace\!\!\vert\thickspace\!\!\vert}
\def\Etilh{{\bar{E}_h}}
\def\PD{P(\partial_t)}

\def\dH#1{\dot H^{#1}(\Omega)}
\def\normh#1#2{\tribar #1 \tribar_{\dot H^{#2}(\Omega)}}
\def\vecal{{\vec{\al}}}
\def\pa{\partial}

\title[Galerkin FEM for multi-term time-fractional diffusion]
{The Galerkin Finite Element Method for A Multi-term Time-Fractional Diffusion equation}
\author {Bangti Jin \and Raytcho Lazarov \and Yikan Liu \and Zhi Zhou}
\address{Department of Mathematics, University of California, Riverside,
900 University Ave., Riverside, CA 92521, USA (bangti.jin@gmail.com)}
\address{Department of Mathematics, Texas A\&M University, College Station, TX 77843-3368, USA \\(lazarov@math.tamu.edu)}
\address{Graduate School of Mathematical Sciences, The University of Tokyo, 3-8-1 Komaba, Meguro,
Tokyo 153, Japan (ykliu@ms.u-tokyo.ac.jp)}
\address{Department of Mathematics, Texas A\&M University, College Station, TX 77843-3368, USA \\(zzhou@math.tamu.edu)}
\date{started August, 2013; today is \today}

\begin{abstract}
We consider the initial/boundary value problem for a diffusion equation
involving multiple time-fractional derivatives on a bounded convex polyhedral domain.
We analyze a space semidiscrete scheme based on the standard Galerkin finite element method
using continuous piecewise linear functions. Nearly optimal error estimates for both cases of initial
data and inhomogeneous term are derived, which cover both smooth and nonsmooth data.
Further we develop a fully discrete scheme based on a finite difference discretization
of the time-fractional derivatives, and discuss its stability and error estimate.
Extensive numerical experiments for one and two-dimension problems confirm the convergence rates
of the theoretical results.
\end{abstract}
\maketitle

\section{introduction}\label{sec:intro}
We consider the following initial/boundary value problem for a multi-term time fractional diffusion
equation in $u(x,t)$:
\begin{alignat}{3}\label{eqn:goveq}
   \PD u-\Delta u&= f,&&\quad \text{in  } \Omega&&\quad T \ge t > 0,\notag\\
   u&=0,&&\quad\text{on}\  \partial\Omega&&\quad T \ge t > 0,\\
    u(0)&=v,&&\quad\text{in  }\Omega,&&\notag
\end{alignat}
where $\Omega$ denotes a bounded convex polygonal domain in $\mathbb R^d\,(d=1,2,3)$ with a boundary
$\partial\Omega$, $f$ is the source term, and the initial data $v$ is a given function on $\Omega$ and
$T>0$ is a fixed value. Here the differential operator $\PD$ is defined by
\begin{equation*}
\PD=\partial_t^{\al}+\sum_{i=1}^m b_i \partial_t^{\al_i},
\end{equation*}
where $0< \al_m  \le ...\le \al_1<\al<1$ are the orders of the fractional derivatives, $b_i > 0$,
$i=1,2,...,m$, with the left-sided Caputo fractional derivative $\partial_t^\beta u$ being defined by
(cf. \cite[pp.\,91]{KilbasSrivastavaTrujillo:2006})
\begin{equation}\label{eqn:McT}
   \partial_t^\beta u(t)= \frac{1}{\Gamma(1-\beta)} \int_0^t(t-\tau)^{-\beta}\frac{d}{d\tau}u(\tau) d\tau,
\end{equation}
where $\Gamma(\cdot)$ denotes the Gamma function.

In the case of $m=0$, the model \eqref{eqn:goveq} reduces to its single-term counterpart
\begin{equation}\label{eqn:single}
  \partial_t^\alpha u -\Delta u = f\quad \mbox{ in } \Omega\times(0,T].
\end{equation}
This model has been studied extensively from different aspects due to its extraordinary
capability of modeling anomalous diffusion phenomena in highly heterogeneous aquifers and
complex viscoelastic materials \cite{AdamsGelhar:1992,Nigmatulin:1986}.
It is the fractional analogue of the classical diffusion equation:
with $\alpha=1$, it recovers the latter, and thus inherits some of its analytical properties.
However, it differs considerably from the latter
in the sense that, due to the presence of the nonlocal fractional
derivative term, it has limited smoothing property in space and slow asymptotic decay in
time \cite{SakamotoYamamoto:2011}, which in turn also impacts related numerical analysis
\cite{JinLazarovZhou:2013} and inverse problems \cite{JinRundell:2012,SakamotoYamamoto:2011}.

The model \eqref{eqn:goveq} was developed to improve the modeling accuracy of the single-term
model \eqref{eqn:single} for describing anomalous diffusion. For example, in
\cite{SchumerBensonMeerschaertBaeumer:2003}, a two-term fractional-order diffusion model was
proposed for the total concentration in solute transport, in order to distinguish explicitly the mobile and
immobile status of the solute using fractional dynamics. The kinetic equation with two fractional derivatives
of different orders appears also quite naturally when describing subdiffusive motion in velocity
fields \cite{MetzlerKlafterSokolov:1998}; see also \cite{KellyMcGoughMeerschaert:2008}
for discussions on the model for wave-type phenomena.

There are very few mathematical studies on the model \eqref{eqn:goveq}. Luchko \cite{Luchko:2011}
established a maximum principle for problem \eqref{eqn:goveq}, and constructed a generalized solution
for the case $f\equiv0$ using the multinomial Mittag-Leffler function. Jiang et al
\cite{JiangLiuTurnerBurrage:2012b} derived analytical solutions for the diffusion equation with
fractional derivatives in both time and space. Li and Yamamoto \cite{LiYamamoto:2013} established
existence, uniqueness, and the H\"{o}lder regularity of the solution using a fixed point argument
for problem \eqref{eqn:goveq} with variable coefficients $\{b_i\}$. Very recently, Li et al
\cite{LiLiuYamamoto:2013} showed the uniqueness and continuous dependence of the solution on the initial value
$v$ and the source term $f$, by exploiting refined properties of the multinomial Mittag-Leffler function.

The applications of the model \eqref{eqn:goveq} motivate the design and analysis of numerical schemes
that have optimal (with respect to data regularity) convergence rates. Such schemes are
especially valuable for problems where the solution has low regularity. The case
$m=0$, i.e., the single-term model \eqref{eqn:single}, has been extensively studied, and stability
and error estimates were provided; see \cite{LinXu:2007,ZhangSunWu:2011} for the finite difference
method, \cite{LiXu:2009,ZayernouriKardiadakis:2014} for the spectral method, \cite{McLeanThomee:2010,Mustapha:2011,
MustaphaMcLean:2013,JinLazarovZhou:2013,JinLazarovPasciakZhou:2013,JinLazarovPasciakZhou:2013a} for the
finite element method, and \cite{BrunnerLingYamamoto:2010,FuChenYang:2013} for meshfree methods based on
radial basis functions, to name a few. In particular, in \cite{JinLazarovPasciakZhou:2013a,JinLazarovPasciakZhou:2013,
JinLazarovZhou:2013}, the authors established almost optimal error estimates with respect to the regularity of
the initial data $v$ and the right hand side $f$ for a semidiscrete Galerkin scheme. These studies include the
interesting case of very weak data, i.e., $ v \in \dH {q}$ and $f\in L^\infty(0,T;\dH q)$ for $-1 < q \le 0$.

Numerical methods for the general multi-term case for an ordinary differential equation were considered in
\cite{Katsikadelis:2009,ElSayedElKallaZiada:2010}. In \cite{ZhaoXiaoXu:2013}, a scheme based on the
finite element method in space and a specialized finite difference method in time was proposed for \eqref{eqn:goveq},
and error estimates were derived. We also refer to  \cite{LiuMeerschaert:2013}
for a numerical scheme based on a fractional predictor-corrector method for the multi-term time fractional
wave-diffusion equation. The error analysis in these works is done under the assumption that
the solution is sufficiently smooth and therefore it excludes the case of low regularity solutions. This
is the main goal of the present study. However, the derivation of optimal with respect to the regularity error estimates
requires additional analysis of the properties of problem  \eqref{eqn:goveq}, e.g.,
stability, asymptotic behavior for $t \to 0^+$. Relevant results of this type have recently been
obtained in \cite{LiLiuYamamoto:2013}, which, however, are not enough for the analysis of the semidiscrete
Galerkin scheme, and hence in Section \ref{sec:prelim}, we make the necessary extensions.

Now we describe the semidiscrete Galerkin scheme. Let ${\{\mathcal{T}_h\}}_{0<h<1}$ be a family of shape regular
and quasi-uniform partitions of the domain $\Omega$ into $d$-simplexes, called finite elements,
with a maximum diameter $h$. The approximate solution $u_h$ is sought in the finite element
space $X_h$ of continuous piecewise linear functions over the triangulation $\mathcal{T}_h $
\begin{equation*}
  X_h =\left\{\chi\in H^1_0(\Omega): \ \chi ~~\mbox{is a linear function over}  ~~\tau,  
 \,\,\,\,\forall \tau \in \mathcal{T}_h\right\}.
\end{equation*}
The semidiscrete Galerkin FEM for problem \eqref{eqn:goveq}
is: find $ u_h (t)\in X_h$ such that
\begin{equation}\label{eqn:fem}
   {( \PD u_{h},\chi)}+ a(u_h,\chi)= {(f, \chi)}, \quad \forall \chi\in X_h,\ T \ge t >0, \quad u_h(0)=v_h,
\end{equation}
where $a(u,w)=(\nabla u, \nabla w) ~~ \text{for}\ u, \, w\in H_0^1(\Omega)$, and $v_h \in X_h$ is an
approximation of the initial data $v$. The choice of $v_h$ will depend on the smoothness of the initial
data $v$. We shall study the convergence of the semidiscrete scheme \eqref{eqn:fem} for the case of
initial data $ v \in \dH q$, $-1<q\leq 2$, and right hand side $f\in L^\infty(0,T;\dH q)$, $-1<q<1$.
The case of nonsmooth data, i.e., $-1<q\leq 0$, is very common in inverse problems and optimal control \cite{JinRundell:2012,
SakamotoYamamoto:2011}; see also \cite{XieZou:2005,JinLu:2012,CasasClasonKunisch:2013,CasasZuazua:2013} for the parabolic counterpart.

The goal of this work is to develop a numerical scheme based on the finite element approximation for
the model \eqref{eqn:goveq}, and provide a complete error analysis. We derive error estimates optimal
with respect to the data regularity for the semidiscrete scheme, and a convergence rate $O(h^2+\tau^{2
-\alpha})$ for the fully discrete scheme in case of a smooth solution. Specifically, our essential
contributions are as follows. First, we obtain an improved regularity result for the inhomogeneous
problem, by allowing less regular source term, cf. Theorem \ref{thm:regeps2}. This is achieved by
first establishing a new result, i.e., the complete monotonicity of the multinomial Mittag-Leffler
function, cf. Lemma \ref{lem:MMLcm}. Second, we derive nearly
optimal error estimates for a semidiscrete Galerkin scheme for both homogeneous and inhomogeneous
problems, cf. Theorems \ref{thm:SG-smooth}-\ref{thm:gal:l2}, which cover both smooth and nonsmooth data.
Third, we develop a fully discrete scheme based on a finite difference method in time, and establish
its stability and error estimates, cf. Theorem \ref{thm:estfull}. We note that the derived
error estimate for the fully discrete scheme holds only for smooth solutions.

The rest of the paper is organized as follows. In Section \ref{sec:prelim}, we recall the solution theory
for the model \eqref{eqn:goveq} for both homogeneous and inhomogeneous problems, using
properties of the multinomial Mittag-Leffler function. The readers not interested in the analysis may
proceed directly to Section \ref{sec:galerkin}. Almost optimal error estimates for their Galerkin
finite element approximations are given in Section \ref{sec:galerkin}. Then a fully discrete scheme based
on a finite difference approximation of the Caputo fractional derivatives is given in Section \ref{sec:fulldis},
and an error analysis is also provided. Finally, extensive numerical experiments
are presented to illustrate the accuracy and efficiency of the Galerkin scheme, and to verify the convergence
theory. Throughout, we denote by $C$ a generic constant, which may differ at different occurrences, but always
independent of the mesh size $h$ and time step size $\tau$.
\section{Solution theory}\label{sec:prelim}
In this part, we recall the solution theory for problem \eqref{eqn:goveq}.
We shall describe the solution representation using the multinomial Mittag-Leffler function,
and derive optimal solution regularity for the homogeneous and inhomogeneous problems.

\subsection{Multinomial Mittag-leffler function}\label{ssec:MultiML}

First we recall the multinomial Mittag-Leffler function, introduced in \cite{HadidLuchko:1996}.
For $0<\beta<2$, $0<\beta_i<1$ and $z_i\in \mathbb C$, $i=1,...,m$,
the multinomial Mittag-Leffler function $E_{(\beta_1,...,\beta_m),\beta}
(z_1,...,z_m)$ is defined by 
\begin{equation*}
    E_{(\beta_1,...,\beta_m),\beta}(z_1,...,z_m)=\sum_{k=0}^{\infty}
    \sum_{\substack{l_1+...+l_m=k\\l_1\ge0,...,l_m\ge0}} (k;l_1,...,l_m)
     \frac{\prod_{i=1}^m z_i^{l_i}}{\Gamma(\beta+\Sigma_{i=1}^m \beta_i l_i)},
\end{equation*}
where the notation $(k;l_1,...,l_m)$ denotes the multinomial coefficient, i.e.,
\begin{equation*}
  (k;l_1,...,l_m)=\frac{k!}{l_1!...l_m!}\quad \mbox{with } k = \sum_{i=1}^ml_i.
\end{equation*}

It generalizes the exponential function $e^z$: with $m=1$ and $\beta=\beta_1=1$, it reproduces
the exponential function $e^z$. It appears in the solution representation of problem
\eqref{eqn:goveq}, cf. \eqref{eqn:solrep} below. We shall need the following two important lemmas on the
function $E_{(\beta_1,...,\beta_m),\beta}(z_1,...,z_m)$, recently obtained in
\cite[Section 2.1]{LiLiuYamamoto:2013}.
\begin{lemma}\label{lem:MLbound}
Let $0<\beta<2$, $0<\beta_i<1$, $\beta_1>\max\{\beta_2,...,\beta_m\}$
and $\frac{\beta_1\pi}{2}<\mu<\beta_1\pi$.
Assume that there is $K >0$ such that $-K\le z_i<0$, $i=2,...,m$.
Then there exists a constant $C=C(\beta_1,...,\beta_m,\beta,K,\mu)>0$
such that
\begin{equation*} 
    E_{(\beta_1,...,\beta_m),\beta}(z_1,...,z_m)\le \frac{C}{1+|z_1|},
        \quad \quad\quad \mu\leq|\mathrm{arg}(z_1)|\leq \pi.
\end{equation*}
\end{lemma}
\begin{lemma}\label{lem:multiMLprop}
Let $0<\beta<2$, $0<\beta_i<1$ and $z_i \in \mathbb C$, $i=1,...,m$.
Then we have
\begin{equation*}
    \frac{1}{\Gamma(\beta_0)}+ \sum_{i=1}^{m}z_iE_{(\beta_1,...,\beta_m),{\beta_0+\beta_i}}(z_1,...,z_m)
    =E_{(\beta_1,...,\beta_m),\beta_0}(z_1,...,z_m).
\end{equation*}
\end{lemma}

\subsection{Solution Representation}\label{ssec:represent}
For $s\ge-1$, we denote by $\dH s\subset H^{-1}(\Omega)$ the Hilbert space induced by the norm:
\begin{equation*}
  \|v\|_{\dH s}^2=\sum_{j=1}^{\infty}\la_j^s \langle v,\fy_j \rangle^2
\end{equation*}
with $\{\la_j\}_{j=1}^\infty$ and $\{\fy_j\}_{j=1}^\infty$ being respectively the eigenvalues and
the $L^2(\Omega)$-orthonormal eigenfunctions of the Laplace operator $-\Delta$ on the domain
$\Omega$ with a homogeneous Dirichlet boundary condition. Then $\{\fy_j\}_{j=1}^\infty$ and $\{\la_j^{1/2}
\fy_j\}_{j=1}^\infty$, form an orthonormal basis in $L^2(\Omega)$ and $H^{-1}(\Omega)$, respectively.
Further, $\|v\|_{\dH 0}=\|v\|_{L^2(\Omega)}=(v,v)^{1/2}$ is the norm in $L^2(\Omega)$ and $\|v\|_{\dH {-1}}
= \|v\|_{H^{-1}(\Omega)}$ is the norm in $H^{-1}(\Omega)$. It is easy to verify that
$\|v\|_{\dH 1}= \|\nabla v\|_{L^2(\Omega)}$ is also the norm in $H_0^1(\Omega)$
and  $\|v\|_{\dH 2}=\|\Delta v\|_{L^2(\Omega)}$ is equivalent to the norm in $H^2(\Omega)\cap H^1_0(\Omega)$
\cite[Lemma 3.1]{Thomee:2006}. Note that $\dH s$, $s\ge -1$ form a
Hilbert scale of interpolation spaces. Hence, we denote $\|\cdot\|_{H^s(\Omega)}$ to
be the norm on the interpolation scale between $H^1_0(\Omega)$ and $L^2(\Omega)$ for $s\in
[0,1]$ and $\|\cdot\|_{H^{s}(\Omega)}$ to be the norm on the interpolation scale between
$L^2(\Omega)$ and $H^{-1}(\Omega)$ for $s\in [-1,0]$.  Then, $\| \cdot \|_{H^s(\Omega)}$
and $\|\cdot\|_{\dH s}$ are equivalent for $s\in [-1,1]$. Further, for a Banach space $B$, we define the space
\begin{equation*}
  L^r(0,T;B) = \{u(t)\in B \mbox{ for a.e. } t\in (0,T) \mbox{ and } \|u\|_{L^r(0,T;B)}<\infty\},
\end{equation*}
for any $r\geq 1$, and the norm $\|\cdot\|_{L^r(0,T;B)}$ is defined by
\begin{equation*}
  \|u\|_{L^r(0,T;B)} = \left\{\begin{aligned}\left(\int_0^T\|u(t)\|_B^rdt\right)^{1/r}, &\quad r\in [1,\infty),\\
     \displaystyle  {\mathrm{esssup}_{t\in(0,T)}}\|u(t)\|_B, &\quad r= \infty.
       \end{aligned}\right.
\end{equation*}

Upon denoting $\vecal=(\al,\al-\al_1,...,\al-\al_m)$, we introduce the following
solution operator
\begin{equation}\label{eqn:opE}
    E(t)v=\sum_{j=1}^{\infty} \left(1-\la_j t^{\al}
    E_{\vecal,1+\al}(-\la_jt^{\al},-b_1t^{\al-\al_1},...,-b_mt^{\al-\al_m})\right) (v,\fy_j)\fy_j.
\end{equation}
This operator is motivated by a separation of variable \cite{LuchkoGorenflo:1999, Luchko:2011}.
Then for problem \eqref{eqn:goveq} with a homogeneous right hand side, i.e., $f\equiv0$, we have $u(x,t)=E(t)v$.
However, the representation \eqref{eqn:opE} is not always very convenient for analyzing its smoothing
property. We derive an alternative representation of the solution operator $E$ using Lemma \ref{lem:multiMLprop}:
\begin{equation}\label{eqn:opE2}
\begin{split}
    E(t)v 
    = &\sum_{j=1}^{\infty}
    E_{\vecal,1} (-\la_jt^{\al},-b_1t^{\al-\al_1},...,-b_mt^{\al-\al_m}) (v,\fy_j)\fy_j\\
      &\ \ + \sum_{i=1}^m b_it^{\al-\al_i}
       \sum_{j=1}^{\infty} E_{\vecal,1+\al-\al_i}
       (-\la_jt^{\al},-b_1t^{\al-\al_1},...,-b_mt^{\al-\al_m}) (v,\fy_j)\fy_j.
\end{split}
\end{equation}

Besides, we define the following operator $\bar{E}$ for $\chi\in L^2(\Omega)$ by
\begin{equation}\label{eqn:Ebar}
    \bar E(t)\chi=\sum_{j=1}^{\infty}  t^{\al-1}
    E_{\vecal,\al}(-\la_jt^{\al},-b_1t^{\al-\al_1},...,-b_mt^{\al-\al_m}) (\chi,\fy_j)\fy_j.
\end{equation}
The operators $E(t) $ and $ {\bar E}(t) $ can be used to represent the solution $u$ of \eqref{eqn:goveq} as:
\begin{equation}\label{eqn:solrep}
   u(t)=E(t)v + \int_0^t  {\bar E}(t-s) f(s) ds.
\end{equation}

The operator $\bar{E}$ has the following smoothing property.
\begin{lemma}\label{lem:barE}
For any $t>0$ and $\chi\in \dH q$, $q\in(-1,2]$, there holds for $0\le p-q \le 2$
\begin{equation*}
     \|\bar E(t) \chi \|_{\dH p} \le Ct^{-1+\al(1+(q-p)/2)}\|\chi\|_{\dH q}.
\end{equation*}
\end{lemma}
\begin{proof}
The definition of the operator $\bar{E}$ in \eqref{eqn:Ebar} and Lemma \ref{lem:MLbound} yield
\begin{equation*}
\begin{split}
\|\bar E(t) \chi \|_{\dH p}^2
      &=t^{-2+(2+q-p)\al}\sum_{j=1}^{\infty} (\la_j t^{\al})^{p-q}
      | E_{\vecal,\al}(-\la_j t^\al,-b_1t^{\al-\al_1},...,-b_mt^{\al-\al_m})|^2 \la_j^q |(\chi,\fy_j)|^2\\
      &\le C t^{-2+(2+q-p)\al}\sum_{j=1}^{\infty} \frac{(\la_j t^{\al})^{p-q}}{(1+\la_jt^{\al})^2} \la_j^q |(\chi,\fy_j)|^2\\
      &\le C t^{-2+(2+q-p)\al}\sum_{j=1}^{\infty} \la_j^q |(\chi,\fy_j)|^2 \le C t^{-2+(2+q-p)\al}\|\chi\|_{\dH q},
\end{split}
\end{equation*}
where the last line follows by the inequality $\sup_{j\in \mathbb{N}}
\frac{ (\la_j t^\al)^{p-q}}{(1+\la_j t^\al)^2}\le C$, for $0\leq p-q\leq 2$.
\end{proof}

\subsection{Solution regularity}
First we recall known regularity results. In \cite{LiYamamoto:2013}, Li and Yamamoto showed that in the
case of variable coefficients $\{b_i(x)\}$, there exists a unique mild solution $u\in C((0,T];
\dH{\gamma})\cap C([0,T];L^2(\Omega))$ and $u\in C([0,T];\dH {\gamma}) \cap L^{\infty}(0,T;\dH 2)$ when
$ v \in L^2(\Omega)$, $f=0$ and $v =0$, $f\in L^{\infty}(0,T]; L^2(\Omega))$, respectively, with
$\gamma\in[0,1)$. These results were recently refined in \cite{LiLiuYamamoto:2013} for the case of constant
coefficients, i.e., problem \eqref{eqn:goveq}. In particular, it was shown that for $v\in \dH q$, $0\leq q
\leq 1$, and $f=0$, $u\in L^{1/(1-q/2)}(0,T;H^2(\Omega)\cap H_0^1(\Omega))$; and for $v=0$ and $f\in L^r(0,T;
\dH q)$, $0\leq q\leq 2$, $r\geq1$, $u\in L^r(0,T;\dH {q+2-\gamma})$ for some $\gamma\in(0,1]$. Here
we follow the approach in \cite{LiLiuYamamoto:2013}, and extend these results to a slightly more general
setting of $v\in \dH q$, $-1<q\leq 2$, and $f\in L^2(0,T;\dH q)$, $-1<q\le 1$. The nonsmooth case, i.e.,
$-1<q\leq 0$, arises commonly in related inverse problems and optimal control problems.

We shall derive the solution regularity to the homogeneous problem, i.e., $f\equiv 0$, and
the inhomogeneous problem, i.e., $v\equiv 0$, separately. These results will be essential for the
error analysis of the space semidiscrete Galerkin scheme in Section \ref{sec:galerkin}.
First we consider the homogeneous problem with initial data $v\in \dH q$, $-1<q\leq 2$.
\begin{theorem}\label{thm:homogreg}
Let $u(t)=E(t)v$ be the solution to problem \eqref{eqn:goveq} with $f\equiv 0$ and $v\in \dH q$,
$q\in(-1,2]$. Then there holds
\begin{equation*}
 \|\PD^\ell u(t) \|_{\dH p}
 \le Ct^{-\al(\ell+(p-q)/2)}\|v\|_{\dH q},\quad t>0,
\end{equation*}
where for $\ell=0$, $ 0 \le p-q\le 2$ and for $\ell=1$, $-2 \le p-q \le 0$.
\end{theorem}
\begin{proof}
We show that \eqref{eqn:opE2} represents indeed the weak solution to problem \eqref{eqn:goveq}
with $f\equiv 0$ and further it satisfies the desired estimate. We first discuss the case
$\ell=0$. By Lemma \ref{lem:MLbound} and \eqref{eqn:opE2} we have for $0\le p-q \le 2$
\begin{equation*}
\begin{split}
    \| E(t)v \|_{\dH p}^2 & = \displaystyle
       \sum_{j=1}^{\infty}\lambda_j^p \Big (E_{\vecal,1} (-\la_jt^{\al},-b_1t^{\al-\al_1},...,-b_mt^{\al-\al_m}) \Big .\\
     & \Big . \quad\quad +\sum_{i=1}^mb_it^{\al-\al_i}
       E_{\vecal,1+\al-\al_i}
       (-\la_jt^{\al},-b_1t^{\al-\al_1},...,-b_mt^{\al-\al_m}) \Big )^2 (v,\fy_j)^2\\
       &\le C t^{-(p-q)\al}\sum_{j=1}^{\infty}\frac{(\la_jt^{\al})^{p-q}}{(1+\la_jt^\al)^2}
         \la_j^q |(v,\fy_j)|^2\le C t^{-(p-q)\al}  \| v \|_{\dH q}^2,
\end{split}
\end{equation*}
where the last line follows from the inequality $\sup_{j\in\mathbb N}\frac{ (\la_jt^\al)^{p-q}
}{(1+\la_jt^\al)^2}\le C$ for $0\le p-q\le 2$. The estimate for the case $\ell=1$ follows from
the identity $\|\PD E(t)v \|_{\dH p} = \| E(t)v \|_{\dH {p+2}}$. It remains to show that
\eqref{eqn:opE2} satisfies also the initial condition in the sense that $\lim_{t\rightarrow 0^+}
\| E(t)v-v  \|_{\dH q}=0$. By identity \eqref{eqn:opE} and Lemma \ref{lem:MLbound}, we deduce
\begin{equation*}
\begin{split}
    \| E(t)v-v  \|_{\dH q}^2
       &=\sum_{j=1}^{\infty} \la_j^{2} t^{2\al}
       \bigg|E_{\vecal,1+\al}(-\la_jt^{\al},-b_1t^{\al-\al_1},...,-b_mt^{\al-\al_m})\bigg|^2\la_j^q| (v,\fy_j)|^2\\
       & \leq C\|v\|_{\dH q}^2 <\infty.
\end{split}
\end{equation*}
Using Lemma \ref{lem:multiMLprop}, we rewrite the term $\la_j t^{\al}E_{\vecal,1+\al}
(-\la_jt^{\al},-b_1t^{\al-\al_1},...,-b_mt^{\al-\al_m})$ as
\begin{equation*}
  \begin{aligned}
      \la_j t^{\al} E_{\vecal,1+\al}&(-\la_jt^{\al},-b_1t^{\al-\al_1},...,-b_mt^{\al-\al_m})\\
    = & (1-E_{\vecal,1} (-\la_jt^{\al},-b_1t^{\al-\al_1},...,-b_mt^{\al-\al_m})) \\
    & \ \ -  \sum_{i=1}^m b_it^{\al-\al_i}E_{\vecal,1+\al-\al_i}
       (-\la_jt^{\al},-b_1t^{\al-\al_1},...,-b_mt^{\al-\al_m}).
  \end{aligned}
\end{equation*}
Upon noting the identity $\lim_{t\to 0^+}(1-E_{\vecal,1} (-\la_jt^{\al},-b_1t^{\al-\al_1},...,
-b_mt^{\al-\al_m}))=0$, and the boundedness of $E_{\vecal,1+\al-\al_i}(-\la_jt^{\al},
-b_1t^{\al-\al_1},...,-b_mt^{\al-\al_m})$ from Lemma \ref{lem:MLbound}, we deduce that for all $j$
\begin{equation*}
  \lim_{t\to0^+}\la_j t^{\al} E_{\vecal,1+\al}(-\la_jt^{\al},-b_1t^{\al-\al_1},...,-b_mt^{\al-\al_m})=0.
\end{equation*}
Hence, the desired assertion follows by Lebesgue's dominated convergence theorem.
\end{proof}

Now we turn to the inhomogeneous problem with a nonsmooth right hand side, i.e.,
$f\in L^{\infty}(0,T;\dH q)$, $-1<q\leq 1$, and a zero initial data $v\equiv0$.
\begin{theorem}\label{thm:inhomogreg}
For $f\in L^{\infty}(0,T;\dH q)$, $-1<q\le1$, and $v\equiv 0$, the solution
$u$ \eqref{eqn:solrep} belongs to $L^{\infty}(0,T;\dH {q+2-\ep})$ for any $\epsilon>0$ and satisfies
\begin{equation}\label{eq:regeps}
    \|u(\cdot,t)\|_{\dH {q+2-\epsilon}} \leq C\epsilon^{-1}t^{\ep\al/2}\|f\|_{L^\infty(0,t;\dH q)}.
\end{equation}
Hence, it is a solution to problem \eqref{eqn:goveq}
with a homogeneous initial data $v=0$.
\end{theorem}
\begin{proof}
By construction, it satisfies the governing equation.
By Lemma \ref{lem:barE}, we have
\begin{equation*}
  \begin{aligned}
    \|u(\cdot,t)\|_{\dH {q+2-\ep}} & = \|\int_0^t \bar{E}(t-s)f(s)ds\|_{\dH {q+2-\ep}}\\
      & \leq \int_0^t \|\bar{E}(t-s)f(s)\|_{\dH {q+2-\ep}} ds \\
      & \leq C\int_0^t (t-s)^{\ep\al/2-1} \|f(s)\|_{\dH q}ds\\
      &\leq C\epsilon^{-1}t^{\ep\al/2}\|f\|_{L^\infty(0,t;\dH q)}
  \end{aligned}
\end{equation*}
which shows the desired estimate. Further, it satisfies
the initial condition $u(0)=0$, i.e., for any $\ep>0$, $\lim_{t\to 0^+}\|u(\cdot,t)
\|_{\dH {q+2-\ep}}=0$, and thus it is indeed a solution of \eqref{eqn:goveq}.
\end{proof}

Next we extend Theorem \ref{thm:inhomogreg} to allow less regular right hand sides $f\in L^2(0,T;\dH q)$,
$-1<q\le 1$. Then the function  $u(x,t)$ satisfies also the
differential equation as an element in the space $L^2(0,T;\dot H^{q+2}(\Omega))$. However, it may not
satisfy the homogeneous initial condition $u(x,0)=0$. In Remark \ref{weakest_sol} below, we argue that the
weakest class of source term that produces a legitimate weak solution of \eqref{eqn:goveq} is $f \in
L^r(0,T;\dH q)$ with $r>1/\al$ and $-1 < q \le 1$. Obviously, for $1/2<\al <1$, it does
give a solution $u(x,t) \in L^2(0,T;\dH {q+2})$.  To this end, we introduce the shorthand notation
\begin{equation*}
  \bar E_{\vecal}^j(t)  = t^{\al-1} E_{\vecal,\al}(-\la_jt^{\al},-b_1t^{\al-\al_1},...,-b_mt^{\al-\al_m}).
\end{equation*}

The function $\bar E_{\vecal}^j(t)$ is completely monotone; see Appendix \ref{app:MMLcm}
for the technical proof.
\begin{lemma}\label{lem:MMLcm}
The function $ \bar E_{\vecal}^j(t)$ for $j\in \mathbb{N}$ has the following properties:
\begin{eqnarray*}
	 \bar E_{\vecal}^j(t)\ \mbox { is completely monotone} \quad \mbox{and}\quad
	\int_0^T | \bar E_{\vecal}^j(t)|\,dt<\frac{1}{\la_j}. 
\end{eqnarray*}
\end{lemma}
\begin{theorem}\label{thm:regeps2}
For $f\in L^2(0,T;\dH q)$, $-1<q\le1$, the representation \eqref{eqn:solrep} belongs
to $L^{2}(0,T;\dH {q+2})$ and satisfies  the {\it a priori} estimate
\begin{equation}\label{eq:regeps2}
    \|u\|_{L^2(0,t;\dH {q+2})} + \| \PD u \|_{L^2(0,t;\dH {q})} \leq C\|f\|_{L^2(0,t;\dH q)}.
\end{equation}
\end{theorem}
\begin{proof}
By Young's inequality for the convolution $\|k\ast f\|_{L^p}\le\|k\|_{L_1}\|f\|_{L^p}$,
$k\in L^1$, $f\in L^p$, $p\ge 1$, and Lemma \ref{lem:MMLcm}, we deduce
\begin{equation*}
   \begin{aligned}
      \left\|\int_0^t \bar E_{\vecal}^n (t-\tau)f_n(\tau)\,d\tau\right\|_{L^2(0,T)}^2&\le
      \left(\int_0^T|\bar E_{\vecal}^n (t)|\,dt\right)^2  \left(\int_0^T |f_n(t)|^2\,dt\right)
      \le \frac{1}{\la_n^2} \int_0^T |f_n(t)|^2\,dt.
   \end{aligned}
\end{equation*}
Hence,
\begin{equation*}
  \begin{aligned}
    \| u\|_{L^2(0,T;\dot H^{q+2}(\Omega))}^2&\le \sum_{n=1}^{\infty}\la_n^{q+2}\left\|\int_0^t
    \bar E_{\vecal}^n (t-\tau)f_n(\tau)\,d\tau\right\|_{L^2(0,T)}^2\\
    &\le \sum_{n=1}^{\infty}\lambda_n^q\int_0^T |f_n(t)|^2\,dt= \|f\|_{L^2(0,T;\dH q)}^2.
  \end{aligned}
\end{equation*}
The estimate on $\| \PD u \|_{L^2(0,t;\dH {q})}$ follows analogously. This completes the proof.
\end{proof}

\begin{remark}\label{weakest_sol}
The condition $f\in  L^\infty(0,T;\dH q)$ in Theorem \ref{thm:inhomogreg} can be weakened
to $f\in L^r(0,T;\dH q)$ with $r>1/\alpha$. This follows from Lemma \ref{lem:barE} and
H\"older's inequality with $r'$, $1/r' + 1/r=1$
\begin{equation*}
  \begin{aligned}
    \|u(\cdot,t)\|_{\dH q} & \leq \int_0^t \|\bar{E}(t-s)f(s)\|_{\dH {q}} ds
       \leq C\int_0^t (t-s)^{\alpha-1} \|f(s)\|_{\dH q}ds \\
     &  \leq C\left(\frac{t^{1+r'(\alpha-1)}}{1+r'(\alpha-1)}\right)^{1/r'}\|f\|_{L^r(0,t;\dot H^q(\Omega))},
  \end{aligned}
\end{equation*}
where $1+r'(\alpha-1)>0$ by the condition $r>1/\alpha$. It follows from this that the initial
condition $u(\cdot,0)=0$ holds in the following sense: $\lim_{t\to0^+}\|u(\cdot,t)\|_{\dH q}=0$.
Hence for any $\alpha\in (1/2,1)$ the representation \eqref{eqn:solrep} remains a
legitimate solution under the weaker condition $f\in L^2(0,T;\dH q)$.
\end{remark}

\section{Error Estimates for Semidiscrete Galerkin Scheme}\label{sec:galerkin}
Now we derive and analyze a space semidiscrete Galerkin finite element scheme. First we
describe the semidiscrete scheme, and then derive almost optimal error estimates for the homogeneous
and inhomogeneous problems separately. In the analysis we essentially use the technique
developed in \cite{JinLazarovZhou:2013} and improved in \cite{JinLazarovPasciakZhou:2013,JinLazarovPasciakZhou:2013a}.

\subsection{Semidiscrete scheme}
To describe the scheme, we need the $L^2(\Omega)$ projection $P_h:L^2(\Omega)\to X_h$ and
Ritz projection $R_h:H^1_0(\Omega)\to X_h$, respectively, defined by
\begin{equation*}
  \begin{aligned}
    (P_h \psi,\chi) & =(\psi,\chi) \quad\forall \chi\in X_h,\\
    (\nabla R_h \psi,\nabla\chi) & =(\nabla \psi,\nabla\chi) \quad \forall \chi\in X_h.
  \end{aligned}
\end{equation*}

The operators $R_h$ and $P_h$ satisfy the following approximation property.
\begin{lemma}\label{lem:prh-bound}
For any $\psi\in \dH q$, $q=1,2$, the operator $R_h$ satisfies:
\begin{eqnarray*}
   \|R_h \psi-\psi\|_{L^2(\Omega)}+h\|\nabla(R_h \psi-\psi)\|_{L^2(\Omega)}\le Ch^q\| \psi\|_{\dot H^q(\Omega)}.
\end{eqnarray*}
Further, for $s\in [0,1]$ we have
\begin{equation*}
  \begin{aligned}
     \|(I-P_h)\psi \|_{H^s(\Omega)} &\le Ch^{2-s} \|\psi\|_{\dH 2}\quad \forall \psi\in
       H^2(\Omega)\cap H^1_0(\Omega),\\
     \|(I-P_h)\psi \|_{H^s(\Omega)} &\le Ch^{1-s} \|\psi\|_{\dH 1}\quad \forall \psi\in H^1_0(\Omega).
  \end{aligned}
\end{equation*}
\end{lemma}

Now we can describe the semidiscrete Galerkin scheme. Upon introducing the discrete Laplacian $\Delta_h: X_h\to X_h$ defined by
\begin{equation*}
  -(\Delta_h\psi,\chi)=(\nabla\psi,\nabla\chi)\quad\forall\psi,\,\chi\in X_h,
\end{equation*}
and $f_h= P_h f$, we may write the spatially discrete problem \eqref{eqn:fem} as
\begin{equation}\label{eqn:fem-operator}
   \PD u_{h}(t)-\Delta_h u_h(t) =f_h(t) \for t\ge0 \quad \mbox{with} \quad  u_h(0)=v_h,
\end{equation}
where $v_h\in X_h$ is an approximation to the initial data $v$. Next we give a solution
representation of \eqref{eqn:fem-operator} using the eigenvalues and eigenfunctions
$\{\la^h_j\}_{j=1}^{N}$ and $\{\fy_j^h\}_{j=1}^{N}$ of the discrete Laplacian
$-\Delta_h$. First we introduce the operators $E_h$ and $\bar{E}_h$, the discrete analogues
of \eqref{eqn:opE2} and \eqref{eqn:Ebar}, for $t>0 $, defined respectively by
\begin{equation}\label{eqn:Eh}
  \begin{split}
    E_h(t)v_h&=\sum_{j=1}^{N}E_{\vecal,1} (-\la_j^ht^{\al},-b_1t^{\al-\al_1},...,-b_mt^{\al-\al_m}) (v,\fy_j^h)\fy_j^h\\
      & + \sum_{i=1}^m b_it^{\al-\al_i} \sum_{j=1}^{N} E_{\vecal,1+\al-\al_i}
       (-\la_j^ht^{\al},-b_1t^{\al-\al_1},...,-b_mt^{\al-\al_m}) (v,\fy_j^h)\fy_j^h,\\
\end{split}
\end{equation}
and
\begin{equation}\label{eqn:Ehbar}
 \Etilh(t) f_h = \sum_{j=1}^N t^{\al-1} E_{\vecal,\al}
 (-\la_j^ht^{\al},-b_1t^{\al-\al_1},...,-b_mt^{\al-\al_m})\,(f_h,\fy^h_j) \, \fy_j^h.
\end{equation}
Then the solution $u_h$ of the discrete problem \eqref{eqn:fem-operator} can be expressed by:
\begin{equation}\label{eqn:Duhamelh}
     u_h(x,t)= E_h(t) v_h + \int_0^t \Etilh(t-s) f_h(s)\,ds.
\end{equation}

On the finite element space $X_h$, we introduce
the discrete norm $\tribar \cdot\tribar_{\dH p}$ defined by
\begin{equation*}
  \tribar \psi\tribar_{\dH p}^2 =
      \sum_{j=1}^N(\la_j^h)^p(\psi,\fy_j^h)^2\quad \psi\in X_h.
\end{equation*}
The norm $\tribar \cdot\tribar_{\dH p}$ is well defined for all real $p$. Clearly,
$\tribar \psi\tribar_{\dH 1}=\|\psi\|_{\dH 1}$
and $\tribar \psi\tribar_{\dH 0}=\|\psi\|_{L^2(\Omega)}$ for any $\psi\in X_h$.
Further, the following inverse inequality holds \cite{JinLazarovZhou:2013}:
if the mesh $\mathcal{T}_h$ is quasi-uniform, then for any $l>s$
\begin{equation}\label{eqn:inverse}
  \tribar \psi\tribar_{\dH l}\le Ch^{s-l}\tribar \psi\tribar_{\dH {s}}\quad \forall \psi\in X_h.
\end{equation}

\begin{lemma}\label{lem:regh}
Assume that the mesh $\mathcal{T}_h$ is quasi-uniform. Then
for any $v_h \in X_h$ the function $u_h(t)=E_h(t)v_h$ satisfies
\begin{equation*}
    \tribar \PD^\ell u_h(t) \tribar_{\dH p}
        \le Ct^{ -\al(\ell + (p-q)/2)}\tribar v_h\tribar_{\dH q}, \quad t >0,
\end{equation*}
where for $\ell=0$, $0\leq p-q\leq 2$ and for $\ell=1$, $p \le q \le p+2$.
\end{lemma}
\begin{proof}
Upon noting $\tribar \PD E_h(t)v_h \tribar_{\dH p}=\tribar  E_h(t)v_h \tribar_{\dH {p+2}}$,
it suffices to show the case $\ell=0$. Using the representation \eqref{eqn:Eh}
and Lemma \ref{lem:MLbound}, we have for $0\leq p-q\leq 2$
\begin{equation*}
\begin{split}
    \tribar E_h(t)v_h \tribar_{\dH p}^2 &\le C \sum_{j=1}^{N}\frac{(\la_j^h)^p}{(1+\la_j^h t^\al)^2} |(v_h,\fy_j^h)|^2\\
       &\le C \sup_{1 \le j\le N}\frac{ (\la_j^h t^\al)^{p-q}}{(1+\la_j^h t^\al)^2}
       t^{-(p-q)\al}\sum_{j=1}^N\frac{(\la_j^h t^{\al})^{p-q}}{(1+\la_j^h t^\al)^2}(\la_j^h)^q |(v_h,\fy_j^h)|^2\\
       &\le C t^{-(p-q)\al}  \tribar v_h \tribar_{\dH q}^2,
\end{split}
\end{equation*}
where the last inequality follows from $\sup_{1 \le j\le N}\frac{
(\la_j^h t^\al)^{p-q}}{(1+\la_j^h t^\al)^2}\le C$ for $0\leq p-q\leq 2$.
\end{proof}

The next result is a discrete analogue to Lemma \ref{lem:barE}.
\begin{lemma}\label{lem:reg-f}
Let $\Etilh$ be defined by \eqref{eqn:Ehbar} and $ \chi \in X_h$. Then for all $t >0$
\begin{equation*}
 \tribar \Etilh(t) \chi \tribar_{\dH p} \le \left \{
\begin{array}{ll}
  Ct^{ -1 + \al(1 + (q -p)/2)}\tribar \chi \tribar_{\dH q}, & \quad 0\le p-q\le 2, \\[1.3ex]
  Ct^{ -1 + \alpha }\tribar \chi \tribar_{\dH q},  & \quad p< q.
\end{array} \right.
\end{equation*}
\end{lemma}
\begin{proof}
The proof for the case $0\le p-q\le2$ is similar to Lemma \ref{lem:barE}. The other assertion
follows from the fact that $\{\lambda_j^h\}_{j=1}^N$ are bounded from zero independent of $h$.
\end{proof}

\subsection{Error estimates for the homogeneous problem}
To derive error estimates, first we consider the
case of smooth initial data, i.e., $v \in \dH 2$. To this end, we split the error
$u_h(t)-u(t)$ into two terms:
\begin{equation*}
  u_h-u= (u_h-R_hu)+(R_hu-u):=\vth + \varrho.
\end{equation*}
By Lemma \ref{lem:prh-bound} and Theorem \ref{thm:homogreg}, we have for any $t>0$
\begin{equation}\label{eqn:rho-bound}
 \| \varrho(t) \|_{L^2(\Omega)} + h \|\nabla\varrho(t)\|_{L^2(\Omega)}\le Ch^2 t^{-(1-q/2)\al}\|v\|_{\dH q} \quad v\in \dH q, q=1,2.
\end{equation}
So it suffices to get proper estimates for $\vth(t) $, which is given below.

\begin{lemma}\label{lem:vth-smooth}
The function $\vth(t):=u_h(t)-R_hu(t)$ satisfies for $p=0,1$
\begin{equation*}
  \|\vth(t)\|_{\dH p} \leq Ch^{2-p}\|v\|_{\dH 2} .
\end{equation*}
\end{lemma}
\begin{proof}
Using the identity $\Delta_hR_h=P_h\Delta$, we note that $\vth $ satisfies
\begin{equation*}
\PD \vth(t) -\Delta_h\vth(t) =-P_h \PD \varrho(t) \for t > 0,
\end{equation*}
with $\vth(0)=0$. By the representation \eqref{eqn:Duhamelh},
\begin{equation*} 
\vth(t)=-\int_0^t \Etilh(t-s) P_h \PD \varrho(s)\,ds.
\end{equation*}
Then by Lemmas \ref{lem:reg-f} and \ref{lem:prh-bound}, and Theorem \ref{thm:homogreg}
we have for $p=0,1$
\begin{equation*}
\begin{split}
  \|\vth(t)\|_{\dH p} &\le \int_0^t \| \Etilh(t-s) P_h \PD \varrho(s) \|_{\dH p} \,ds  \\
    & \le C \int_0^t (t-s)^{(1-p/2)\al-1} \|\PD \varrho(s) \|_{L^2(\Omega)}\,ds \\
    & \le C h^{2-p}\int_0^t (t-s)^{(1-p/2)\al-1} \|\PD u(s) \|_{\dH {2-p}}\,ds \\
    & \le C h^{2-p} \int_0^t (t-s)^{(1-p/2)\al-1} s^{-(1-p/2)\al} \,ds \|v\|_{\dH 2} \le Ch^{2-p}\|v\|_{\dH 2},
\end{split}
\end{equation*}
which is the desired result.
\end{proof}

Using \eqref{eqn:rho-bound}, Lemma \ref{lem:vth-smooth} and the triangle inequality,
we arrive at our first estimate, which is formulated in the following Theorem:
\begin{theorem}\label{thm:SG-smooth}
Let $v\in \dH 2$ and $f\equiv 0$, and $u$ and $u_h$ be the solutions of \eqref{eqn:goveq} and \eqref{eqn:fem} with $v_h=R_hv$, respectively. Then
\begin{equation*}
  \|u_h(t)- u(t)\|_{L^2(\Omega)} + h\|\nabla (u_h(t) - u(t))\|_{L^2(\Omega)} \le Ch^2 \|v\|_{\dH 2}.
\end{equation*}
\end{theorem}


Now we turn to the nonsmooth case, i.e., $v\in \dH q$ with $-1< q \leq  1$. Since the Ritz
projection $R_h$ is not well-defined for nonsmooth data, we use instead the $L^2
(\Omega)$-projection $v_h=P_hv$ and split the error $u_h-u$ into:
\begin{equation}\label{eqn:splitnonsmooth}
  u_h-u=(u_h-P_hu)+(P_hu-u):=\vtht + \rlh.
\end{equation}
By Lemma \ref{lem:prh-bound} and Theorem \ref{thm:homogreg} we have for $-1\leq q\leq 1$
\begin{equation*} 
  \| \rlh(t) \|_{L^2(\Omega)} + h \|\nabla\rlh(t) \|_{L^2(\Omega)}
  \leq Ch^2\|u(t)\|_{\dH 2}
   \leq Ch^2 t^{-\al(1-q/2)}\|v\|_{\dH q}. 
\end{equation*}
Thus, we only need to estimate the term $\vtht(t)$, which is stated in the following lemma.
\begin{lemma}\label{lem:vtht}
Let $\vtht(t)=u_h(t)-P_hu(t)$. Then for $p=0,1$, $-1< q\le1$,
there holds (with $\ell_h=|\ln h|$)
\begin{equation*}
  \|\vtht(t)\|_{\dH p}\leq Ch^{\min(q,0)+2-p}\ell_h
    t^{-\al\left(1-\max(q/2,0)\right)} \| v \|_{\dH q}.
\end{equation*}
\end{lemma}
\begin{proof}
Obviously, $ P_h \PD \rlh = \PD P_h(P_hu-u)=0$ and using the identity $\Delta_hR_h=P_h\Delta$,
we get the following problem for $\vtht$:
\begin{equation} \label{eqn:vtht2}
 \PD \vtht(t) -\Delta_h \vtht(t) =
- \Delta_h (R_h u - P_h u)(t), \quad t>0, \quad \vtht(0)=0.
\end{equation}
Using \eqref{eqn:Ehbar}, $\vtht(t)$ can be represented by
\begin{equation}\label{eqn:vtht}
  \vtht(t) = - \int_0^t\Etilh(t-s)\Delta_h(R_hu-P_hu)(s)\,ds.
\end{equation}
Let $A=\Etilh(t-s)\Delta_h(R_hu-P_hu)(s)$. Then by Lemma \ref{lem:regh},
there holds for $p=0,1$:
\begin{equation*}
\begin{split}
 \| A \|_{\dH p}
   &\le C (t-s)^{\ep\al/2-1}\tribar \Delta_h(R_hu-P_hu)(s) \tribar_{\dH {p-2+\ep}} \\
   &\le C (t-s)^{\ep\al/2-1}\tribar(R_hu-P_hu)(s) \tribar_{\dH {p+\ep}}.\\
\end{split}
\end{equation*}
Then by \eqref{eqn:inverse}, Theorem \ref{thm:homogreg}, Lemma
\ref{lem:prh-bound} we have for $p=0,1$ and $-1\le q\le 1$
\begin{equation*}
\begin{split}
 \| A \|_{\dH p}
    &\le Ch^{\min(q,0)+2-p-\ep}(t-s)^{\ep\al/2-1}\| u(s) \|_{\dH {\min(q,0)+2}}\\
    &\le Ch^{\min(q,0)+2-p-\ep}(t-s)^{\ep\al/2-1}
    s^{-\left(1-\max(q/2,0)\right)\al}\| v \|_{\dH {q}}.
\end{split}
\end{equation*}
Then plugging the estimate into \eqref{eqn:vtht} yields
\begin{equation*}
\begin{split}
 \| \vtht \|_{\dH p}
    &\le Ch^{\min(q,0)+2-p-\ep} \int_0^t (t-s)^{\ep\al/2-1}
    s^{-\left(1-\max(q/2,0)\right)\al} \, ds \| v \|_{\dH {q}}\\
    &\le C\epsilon^{-1}h^{\min(q,0)+2-p-\ep} t^{-\al\left(1-\max(q/2,0)\right)}\| v \|_{\dH q}.
\end{split}
\end{equation*}
Now with the choice $\ep=1/\ell_h$, we obtain the desired estimate.
\end{proof}

Now the triangle inequality yields an error estimate for nonsmooth initial data.
\begin{theorem}\label{thm:SG-nonsmooth}
Let $f\equiv 0$, $u$  and $u_h$ be the solutions of \eqref{eqn:goveq} with $v\in \dH q$, $-1<q\le 1$, and \eqref{eqn:fem} with
$v_h=P_hv$, respectively. Then with $\ell_h=|\ln h|$, there holds
\begin{equation*}
 \| u_h(t) - u(t) \|_{L^2(\Omega)} +   h  \|\nabla(u_h(t) - u(t))\|_{L^2(\Omega)}
  \le Ch^{\min(q,0)+2} \, \ell_h \,t^{-\al(1-\max(q/2,0))}\|v\|_{\dH q}.
\end{equation*}
\end{theorem}

\subsection{Error estimates for the inhomogeneous problem}
Now we derive error estimates for the semidiscrete Galerkin approximations of the inhomogeneous problem
with $f\in L^{\infty}(0,T;\dH q)$, $-1<q\leq 0$, and $v\equiv0$, in both $L^2$ and $L^\infty$-norm in time.
To this end, we appeal again to the splitting \eqref{eqn:splitnonsmooth}. By Theorem
\ref{thm:inhomogreg} and Lemma \ref{lem:prh-bound}, the following estimate holds for $\rlh$:
\begin{equation*}
 \| \rlh(t) \|_{L^2(\Omega)} + h \|\nabla\rlh(t)\|_{L^2(\Omega)} \le Ch^{2+q-\epsilon} \|u(t)\|_{\dH {2+q-\epsilon}}
          \le C\epsilon^{-1}h^{2+q-\epsilon}\|f\|_{L^\infty(0,t;\dH q)}.
\end{equation*}
Now the choice $\ell_h= |\ln h|,\, \epsilon=1/\ell_h$, yields
\begin{equation}\label{eqn:Ph-bound}
 \| \rlh(t) \|_{L^2(\Omega)} + h \|\nabla\rlh(t)\|_{L^2(\Omega)}\le C\ell_hh^{2+q} \|f\|_{L^\infty(0,t;\dH q)}.
\end{equation}

Thus, it suffices to bound the term $\vtht$; see the lemma below.
\begin{lemma}\label{lem:vtht-f}
Let $\vtht(t)$ be defined by \eqref{eqn:vtht}, and $f\in L^\infty(0,T;\dH q)$, $-1<q\leq 0$. Then
with $\ell_h=|\ln h|$, there holds
\begin{equation*}
  \|\vtht(t)\|_{L^2(\Omega)}+h\|\nabla\vtht(t)\|_{L^2(\Omega)}\leq Ch^{2+q}\ell_h^2\|f\|_{L^\infty(0,t;\dH q)}.
\end{equation*}
\end{lemma}
\begin{proof}
By \eqref{eqn:Duhamelh} and Lemma \ref{lem:reg-f},
we deduce that for $p=0,1$
\begin{equation*}
   \begin{aligned}
     \|\vtht(t)\|_{\dH p} &\leq \int_0^t \|\Etilh(t-s)\Delta_h(R_h u-P_hu)(s)\|_{\dH p} ds\\
      &\leq C\int_0^t (t-s)^{\epsilon\al/2 - 1} \tribar\Delta_h(R_h u -P_hu)(s)\tribar_{\dH {p-2+\epsilon}}ds\\
      & \leq C\int_0^t(t-s)^{\epsilon\al/2-1} \tribar R_h u(s) -P_hu(s)\tribar_{\dH {p+\epsilon}}ds.
    \end{aligned}
\end{equation*}
Further, using \eqref{eqn:inverse} and Lemma \ref{lem:prh-bound}, we deduce for $p=0,1$
\begin{equation*}
   \begin{aligned}
     \|\vtht(t)\|_{\dH p} & \leq C h^{-\epsilon} \int_0^t(t-s)^{\epsilon\al/2-1} \| R_h u(s) -P_hu(s)\|_{\dH p}ds \\
       &  \leq C h^{2+q-p-2\epsilon} \int_0^t(t-s)^{\epsilon\al/2-1} \| u(s) \|_{\dH {2+q-\epsilon}}ds.
   \end{aligned}
\end{equation*}
Now by \eqref{eq:regeps} and the choice $\ep=1/\ell_h$ we get for $p=0,1$:
\begin{equation*}
   \begin{aligned}
    \|\vtht(t)\|_{\dH p} & \leq C \epsilon^{-1} h^{2+q-p-2\epsilon}  \|f\|_{L^\infty(0,t;\dH q)}
    \int_0^t(t-s)^{\epsilon\al/2-1}t^{\epsilon \alpha/2}ds\\
  &\leq C \epsilon^{-2} h^{2+q-p-2\epsilon}\|f\|_{L^\infty(0,t;\dH q)}
   \le Ch^{2+q-p}\ell_h^2\|f\|_{L^\infty(0,t;\dH q)}.
   \end{aligned}
\end{equation*}
This completes the proof of the lemma.
\end{proof}
An inspection of the proof of Lemma \ref{lem:vtht-f} indicates that for $0<q <1$, one
can get rid of one factor $\ell_h$. Now we can state an error estimate in $L^\infty$-norm in time.
\begin{theorem}\label{thm:gal:linf}
Let $v\equiv0$, $f\in L^\infty(0,T;\dH q)$, $-1<q\leq0$, and $u$ and $u_h$ be the solutions
of \eqref{eqn:goveq} and \eqref{eqn:fem} with $f_h=P_hf$, respectively.
Then with $\ell_h =| \ln h|$ and $t > 0$, there holds
\begin{equation*}
 \| u_h(t) - u(t) \|_{L^2(\Omega)} +
  h\|\nabla(u_h(t) - u(t))\|_{L^2(\Omega)} \le Ch^{2+q} \ell_h^{2} \|f\|_{L^\infty(0,t;\dH q)}.
\end{equation*}
\end{theorem}


Last, we derive an error estimate in $L^2$-norm in time. To this end, we need a discrete analogue
of Theorem \ref{thm:regeps2}, which follows from the identical proof.
\begin{lemma}\label{lem:reg-d-l2}
Let $u_h$ be the solution of \eqref{eqn:fem} with $v_h=0$. Then for arbitrary $p>-1$
\begin{equation*}
   \int_0^T\normh{ \PD u_h(t)}{p}^2 +\normh{ u_h(t) }{p+2}^2 \, dt \le  \int_0^T \normh {f_h(t)}{p}^2 dt.
\end{equation*}
\end{lemma}
\begin{theorem}\label{thm:gal:l2}
Let $v\equiv 0$, $f\in L^\infty(0,T;\dH q)$, $-1<q\leq 0$, and $u$ and $u_h$ be the
solutions of \eqref{eqn:goveq} and \eqref{eqn:fem} with
$f_h=P_hf$, respectively. Then
\begin{equation*}
 \| u_h - u \|_{L^2(0,T;L^2(\Omega))} + h\|\nabla(u_h- u)\|_{L^2(0,T;L^2(\Omega))} \le Ch^{2+q}  \|f\|_{L^2(0,T;\dH q)}.
\end{equation*}
\end{theorem}
\begin{proof}
We use the splitting \eqref{eqn:splitnonsmooth}.
By Theorem \ref{thm:regeps2} and Lemma \ref{lem:prh-bound}
\begin{equation*}
\begin{aligned}
    \| \rlh \|_{L^2(0,T;L^2(\Omega))} + h  \|\nabla \rlh\|_{L^2(0,T;L^2(\Omega))} &\le Ch^{2+q}  \|u\|_{L^2(0,T;\dH {2+q})}\\
                  & \le Ch^{2+q}  \|f\|_{L^2(0,T;\dH q)}.
\end{aligned}
\end{equation*}
By \eqref{eqn:Duhamelh}, \eqref{eqn:vtht2} and Lemmas \ref{lem:reg-d-l2} and \ref{lem:prh-bound}, we have for $p=0,\,1$:
\begin{equation*}
 \begin{split}
   \int_0^T \|\vtht(t)\|^2_{\dH p} dt & \le C\int_0^T \normh{\Delta_h (R_h u - P_h u)(t) }{p-2}^2 dt \\
                    & \le C\int_0^T \normh{(R_h u - P_h u)(t)}{p}^2 dt\\
                    & \le C h^{4+2q-2p} \| u(t) \|_{L^2(0,T;\dH {2+q})}^2\\
                    & \le C h^{4+2q-2p}\| f(t) \|_{L^2(0,T;\dH q)}^2.
 \end{split}
\end{equation*}
Combing the preceding two estimates yields the desired assertion.
\end{proof}

\section{A Fully Discrete Scheme}\label{sec:fulldis}
Now we describe a fully discrete scheme for problem \eqref{eqn:goveq} based on the finite difference
method introduced in \cite{LinXu:2007}. To discretize the time-fractional derivatives, we divide the
interval $[0,T]$ uniformly with a time step size $\tau=T/K$, $K\in\mathbb{N}$.
We use the following discretization:
\begin{equation}\label{eqn:difference}
\begin{aligned}
    \frac{\pa^\al u(x,t_{n+1})}{\partial t^\al} &= \frac{1}{\Gamma(1-\al)}
    \sum_{j=0}^{n} \int_{t_j}^{t_{j+1}}(t_{n+1}-s)^{-\al}\frac{\partial u(x,s)}{\partial s}\,ds \\
    &=  \frac{1}{\Gamma(1-\al)}
    \sum_{j=0}^{n} \frac{u(x,t_{j+1})-u(x,t_j)}{\tau} \int_{t_j}^{t_{j+1}}(t_{n+1}-s)^{-\al}\,ds +r_{\al,\tau}^{n+1}\\
    &= \frac{1}{\Gamma(2-\al)} \sum_{j=0}^{n} d_{\al,j} \frac{u(x,t_{n+1-j})-u(x,t_{n-j})}{\tau^\al}+ r_{\al,\tau}^{n+1},
\end{aligned}
\end{equation}
where $d_{\al,j}=(j+1)^{1-\al}-j^{1-\al}$ with $j=0,1,2,...,n$ and
$r_{\al,\tau}^{n+1}$ denotes the local truncation error, which is given by
\begin{equation*}\label{eqn:residue}
\begin{split}
   | r_{\al,\tau}^{n+1} | \le C \max_{0\le t \le T} |u_{tt}(x,t)|
   \sum_{j=1}^{n} \int_{t_j}^{t_{j+1}}\frac{2s-t_j-t_{j+1}}{({t_{n+1}-s})^{\al}}\,ds +O(\tau^2).
\end{split}
\end{equation*}
Lin and Xu \cite[Lemma 3.1]{LinXu:2007} showed that the truncation error
$r_{\alpha,\tau}^{n+1}$ can be bounded by
\begin{equation}\label{eqn:residueerror}
   | r_{\al,\tau}^{n+1} | \le C \max_{0\le t \le T} \left |u_{tt}(x,t) \right | \tau^{2-\al}.
\end{equation}
Then the multi-term fractional derivative $\PD u(t)$ at $t=t_{n+1}$ in \eqref{eqn:goveq} can be discretized by
\begin{equation}\label{eqn:difference2}
    \PD u(t_{n+1}) = P_\tau (\bar \pa_t) u(t_{n+1})+R_\tau^{n+1},
\end{equation}
where the discrete differential operator $P_\tau(\bar\pa_t)$ is defined by
\begin{equation}\label{eqn:discdiff}
{\small
   P_\tau (\bar \pa_t) u(t_{n+1}) := \frac{1}{\Gamma(2-\al)} \sum_{j=0}^{n}  P_j
   \frac{u(x,t_{n+1-j})-u(x,t_{n-j})}{\tau^\al},}
\end{equation}
where the coefficients $\{P_j\}$ are defined by
$$P_j=d_{\al,j} + \sum_{i=1}^m\frac{\Gamma(2-\al)b_id_{\al_i,j}\tau^{\al-\al_i}}{\Gamma(2-\al_i)},
\quad j\in \mathbb{N}.
$$
Then by \eqref{eqn:residueerror} the local truncation error $R_\tau^{n+1}$ of the approximation
$P_\tau(\bar\partial_t)u(t_{n+1})$ is bounded by
\begin{equation}\label{eqn:residue2}
\begin{split}
   | R_\tau^{n+1} | \le C \max_{0\le t \le T} \left |u_{tt}(x,t) \right | 
   \left(\tau^{2-\al} + \sum_{i=1}^m b_i \tau^{2-\al_i}\right)\le C \tau^{2-\al}\max_{0\le t \le T}
    \left |u_{tt}(x,t) \right |. 
\end{split}
\end{equation}
By the monotonicity and convergence of $\{d_{\al,j}\}$ \cite[equation (13)]{LinXu:2007}, we know that
\begin{equation}\label{eqn:Pmono}
P_0 > P_1 >...>0 \quad \text{and} \quad P_j \rightarrow 0 \quad \text{for} \quad j \rightarrow \infty.
\end{equation}

Now we arrive at the following fully discrete scheme: find $U^{n+1}\in X_h$ such that
\begin{equation}\label{eqn:fully1}
    (P_\tau(\bar \pa_t)U^{n+1},\chi) + (\nabla U^{n+1},\nabla \chi) = (F^{n+1}, \chi)\quad \forall \chi\in X_h,
\end{equation}
where $F^{n+1}=f(x,t_{n+1})$. Upon setting $\gamma=\Gamma(2-\al)\tau^{\al}$, the fully discrete scheme
\eqref{eqn:fully1} is equivalent to finding $U^{n+1}\in X_h$ such that for all $\chi \in X_h$
\begin{equation}\label{eqn:fully}
    P_0(U^{n+1},\chi) + \gamma(\nabla U^{n+1}, \nabla \chi)
    =\sum_{j=0}^{n-1} (P_j-P_{j+1})(U^{n-j},\chi)+P_n (U^0 ,\chi) + \gamma(F^{n+1},\chi).
\end{equation}

The next result gives the stability of the fully discrete scheme.
\begin{lemma}\label{lem:fullystab}
The fully discrete scheme \eqref{eqn:fully} is unconditionally stable, i.e., for all $n\in \mathbb{N}$
\begin{equation}\label{eqn:fullystab}
\| U^n \|_{L^2(\Omega)} \le  \| U^0 \|_{L^2(\Omega)}+  c\max_{1\le j\le n} \|  F^{j}  \|_{L^2(\Omega)}.
\end{equation}
where the constant $c$ depends only on $\al$ and $T$.
\end{lemma}
\begin{proof}
The case $n=1$ is trivial. Then the proof proceeds by mathematical induction. By noting
the monotone decreasing property of the sequence $\{P_n\}$  from \eqref{eqn:Pmono} and choosing $\chi=U^{n+1}$ in
\eqref{eqn:fully}, we deduce
\begin{equation*}
\begin{split}
   P_0 \| U^{n+1}\|_{L^2(\Omega)} &\le \sum_{j=0}^{n-1} (P_j-P_{j+1})  \| U^{n-j} \|_{L^2(\Omega)}
   +  P_n \|  U^0 \|_{L^2(\Omega)}
   +  \gamma \|  F^{n+1}  \|_{L^2(\Omega)} \\
   &\le \sum_{j=0}^{n-1} (P_j-P_{j+1})  \| U^{n-j} \|_{L^2(\Omega)}
   +  P_n \|  U^0 \|_{L^2(\Omega)}
   +  \gamma \max_{1\le j\le n+1} \|  F^{j}  \|_{L^2(\Omega)} \\
   &\le P_0 \| U^0 \|_{L^2(\Omega)}+ \left(c(P_0-P_n)+\gamma\right)\max_{1\le j\le n+1} \|  F^{j}  \|_{L^2(\Omega)} \\
\end{split}
\end{equation*}
Using the monotonicity of $\{P_n\}$ again gives
\begin{equation*}
  c(P_0-P_n)+\gamma \le cP_0-(c P_{N} - \gamma).
\end{equation*}
It suffices to choose a constant $c$ such that $c P_N - \gamma>0$. By taking
$\tau=T/N$, we get
\begin{equation*}
\begin{split}
P_N &= (N+1)^{1-\al}-N^{1-\al}=((T+\tau)^{1-\al}-T^{1-\al})\tau^{\al-1}
    \leq (1-\al)T^{-\al}\tau^{\al} 
\end{split}
\end{equation*}
upon noting the concavity of the function $g(\tau)=(T+\tau)^{1-\alpha}$.
Then by choosing $c=\Gamma(2-\al)T^\al/(1-\al)$ we obtain
\begin{equation*}
\begin{split}
   P_0 \| U^{n+1}\|_{L^2(\Omega)}
   \le P_0 \| U^0 \|_{L^2(\Omega)}+ cP_0\max_{1\le j\le n+1} \|  F^{j}  \|_{L^2(\Omega)}. \\
\end{split}
\end{equation*}
The desired result follows by dividing both sides by $P_0$.
\end{proof}

Next we state an error estimate for the fully discrete scheme.
In order to analyze the temporal discretization error, we
assume the solution is sufficiently smooth.
\begin{theorem}\label{thm:estfull}
Let the solution $u$ be sufficiently smooth, and $\{U^n\}\subset X_h$ be the solution of
the fully discrete scheme \eqref{eqn:fully} with $U^0$ such that
\begin{equation*}
\| U^0-v\|_{L^2(\Omega)} \le Ch^2\| v\|_{\dH 2}.
\end{equation*}
Then there holds
\begin{equation*} 
\begin{split}
\| u(t_n)-U^n\|_{L^2(\Omega)} \le  C \bigg(h^2 (\|v\|_{H^2(\Omega)}  +\| f \|_{L^\infty(0,T;L^2(\Omega))} & +
\max_{0< t\le t_n} \| u_t \|_{\dH 2} )\\
&  +\tau^{2-\al} \max_{0< t\le t_n}\| u_{tt}(t)\|_{L^2(\Omega)} \bigg).
\end{split}
\end{equation*}
\end{theorem}
\begin{proof}
We split the error $e^n=u(t_n)-U^n$ into
\begin{equation*}
  e^n = (u(t_n)-R_h u(t_n))+(R_h u(t_n)- U^n)=: \varrho^n+\theta^n.
\end{equation*}
The term $\varrho^n$ can be bounded by
$$\| u(t_n)-R_h u(t_n) \|_{L^2(\Omega)}\le Ch^2\| u(t_n)\|_{\dH 2}
\le Ch^2(\|v\|_{H^2(\Omega)}+\| f \|_{L^\infty(0,T;L^2(\Omega))}).$$
It suffices to bound the term $\theta^n$.
By comparing \eqref{eqn:goveq} and \eqref{eqn:fully1}, we have the error equation
\begin{equation}\label{eqn:erroreq}
  (P_\tau(\bar \pa_t)\theta^n,\chi) + (\nabla\theta^n,\nabla\chi)=(\omega^n,\chi),
\end{equation}
where the right hand side $\omega^n$ is given by
\begin{equation*}
  \omega^n = R_h P_\tau(\bar \pa_t)u(t_n)-\PD u(t_n)= -P_\tau(\pa_t)\varrho(t_n)-R_\tau^n
  :=\omega_1^n+\omega_2^n,
\end{equation*}
where the truncation error $R_\tau^n$ is defined in \eqref{eqn:difference2}. Using the identity
\begin{equation*}
\varrho(x,t_{j+1})-\varrho(x,t_{j})=\int_{t_{j}}^{t_{j+1}} \varrho_t(x,t) \,dt,
\end{equation*}
we can bound the term $\omega_1^n$ by
\begin{equation*}
\begin{split}
    \| \omega_1^n\|_{L^2(\Omega)}
    &\le C \bigg|\hspace{-0.6mm}\bigg| \sum_{j=0}^{n-1} \frac{\varrho(t_{j+1})-\varrho(t_j)}{\tau} \int_{t_j}^{t_{j+1}}(t_{n}-s)^{-\al}+\sum_{i=1}^m b_i (t_{n}-s)^{-\al_i}\,ds\bigg|\hspace{-0.6mm}\bigg|_{L^2(\Omega)}\\
    &\leq C\sum_{j=0}^{n-1} \tau^{-1} \int^{t_{j+1}}_{t_{j}} \| \varrho_t(t) \|_{L^2(\Omega)} \,dt\int_{t_j}^{t_{j+1}}(t_{n}-s)^{-\al}+\sum_{i=1}^m b_i (t_{n}-s)^{-\al_i}\,ds\\
    &\leq C h^2\max_{0< t\le t_n} \| u_t \|_{\dH 2}
    \left(  \int_0^{t_n}(t_{n}-s)^{-\al}+\sum_{i=1}^m b_i (t_{n}-s)^{-\al_i}\,ds\right)\\
    &\le C h^2\max_{0< t\le t_n} \| u_t \|_{\dH 2}.
\end{split}
\end{equation*}
Meanwhile, the second term $\omega_2^n$ can be bounded using \eqref{eqn:residue2}. Then by the stability from
Lemma \ref{lem:fullystab} for the error equation \eqref{eqn:erroreq}, we obtain
\begin{equation*}
\begin{aligned}
  \| \theta^n  \|_{L^2(\Omega)} & \le C\bigg( \| \theta^0  \|_{L^2(\Omega)} + \max_{1\le j\le n} \| \omega_1^j \|_{L^2(\Omega)}
  + \max_{1\le j\le n} \| \omega_2^j \|_{L^2(\Omega)}\bigg)\\
      & \le C \bigg(h^2\| v \|_{\dH 2}
     + h^2 \max_{0< t\le t_n} \| u_t \|_{\dH 2} +
     \tau^{2-\al} \max_{0< t\le t_n} \|u_{tt}(t)\|_{L^2(\Omega)} \bigg).
\end{aligned}
\end{equation*}
\end{proof}

\begin{remark}
The error estimate in Theorem \ref{thm:estfull} holds only if
the solution $u$ is sufficiently smooth. There seems no known error estimate
expressed in terms of the initial data (and right hand side) only for
fully discrete schemes for nonsmooth initial data even for the
single-term time-fractional diffusion equation with a Caputo fractional derivative.
\end{remark}

\section{Numerical Experiments}\label{sec:numeric}
In this part we present one- and two-dimensional numerical experiments to verify
the error estimates in Sections \ref{sec:galerkin} and \ref{sec:fulldis}. We shall
discuss the cases of a homogeneous problem and an inhomogeneous problem separately.
\subsection{The case of a smooth solution}
Here we consider the following one-dimensional problem on the unit interval $\Omega=(0,1)$ with $0<\beta<\al<1$
\begin{equation}\label{1Dnum}
\begin{split}
       \partial_t^\al u(x,t) +  \partial_t^\beta u(x,t) -\partial_{xx}^2 u(x,t)&=f,  \quad 0< x <1 \quad 0\le t\le T,\\
       u(0,t)=u(1,t)&=0, \quad  0\le t \le T, \\
       u(x,0)&=v(x), \quad  0\le x \le 1.
\end{split}
\end{equation}
In order to verify the estimate in Theorem \ref{thm:estfull}, we first check the case
that the solution $u$ is sufficiently smooth. To this end, we set initial data $v$ to
$v(x)=x(1-x)$ and the source term $f$ to $f(x,t)=(2t^{2-\al}/{\Gamma(3-\al)}+2t^{2-\beta}
/{\Gamma(3-\beta)})(-x^2+x)+2(1+t^2)$. Then the exact solution $u$ is given by $u(x,t)=
(1+t^2)(-x^2+x)$, which is very smooth.

In our computation, we divide the unit interval $\Omega$ into $M$ equally spaced
subintervals, with a mesh size $h = 1/N$. Similarly, we fix the time step size at
$\tau=1/K$. Here we choose $N$ large enough so that the space discretization error is
negligible, and the time discretization error dominates. We measure the accuracy of
the numerical approximation $U^n$ by the normalized $L^2$ error $\| U^n-u(t_n) \|_{L^2(\Omega)}/
\| v \|_{L^2(\Omega)}$. In Table \ref{tab:smooth-sol-time}, we show the temporal
convergence rates, indicated in the column \texttt{rate} (the number in bracket is the
theoretical rate), for three different $\alpha$ values, which fully confirm the theoretical
result, cf. also Fig. \ref{fig:time_error} for the plot of the convergence rates.

\begin{table}[htb!]
\caption{Numerical results for the case with a smooth solution at $t=1$ with $\beta=0.2$
and $\al=0.25, 0.5, 0.95$, discretized on a uniform mesh with $h= 2^{-10}$ and
$\tau =0.2\times2^{-k}$.} \label{tab:smooth-sol-time}
\begin{center}
\begin{tabular}{@{}|c|c|ccccc|c|@{}}
     \hline
      $\al$ & $\tau$ & $1/10$ & $1/20$ & $1/40$ &$1/80$ & $1/160$ & rate\\
     \hline
     $\al=0.25$ & $L^2$-norm & 5.58e-4 &1.73e-4  &5.25e-5  &1.51e-5  &3.90e-6  & $\approx 1.81$ ($1.75$) \\
     \hline
     $\al=0.5$ & $L^2$-norm  & 1.45e-3 &5.11e-4  &1.78e-4  &6.17e-5  &2.08e-5 & $\approx 1.55$ ($1.50$)  \\
     \hline
     $\al=0.95$ & $L^2$-norm & 7.92e-3 &3.79e-3  &1.82e-3  &8.73e-4  &4.20e-4 & $\approx 1.06$ ($1.05$)\\
     \hline
     \end{tabular}
\end{center}
\end{table}
\vspace{-.6cm}
\begin{figure}[htb!]
\center
  \includegraphics[trim = 1cm .1cm 2cm 0.0cm, clip=true,width=10cm]{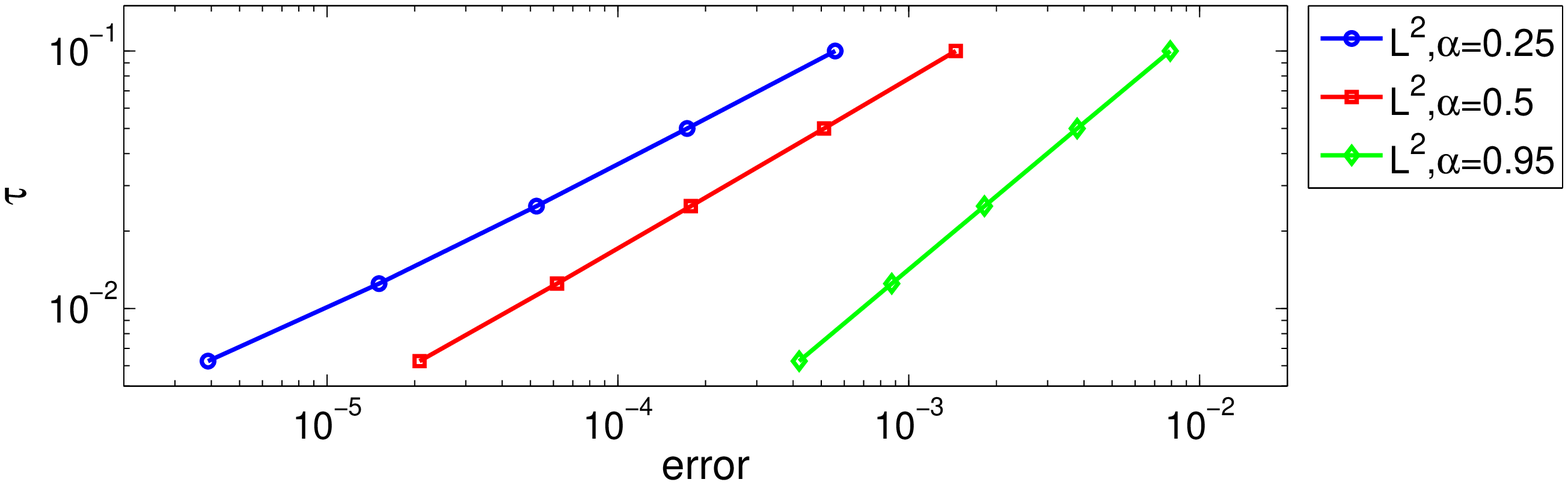}
  \caption{Numerical results for the case with a smooth solution at $t=1$ with $\beta=0.2$
and $\al=0.25, 0.5, 0.95$, discretized on a uniform mesh $h = 2^{-10}$ and $\tau =0.2\times2^{-k}$.}\label{fig:time_error}
\end{figure}
\subsection{Homogeneous problems}
In this part we present numerical results to illustrate the spatial convergence rates in Section
\ref{sec:galerkin}. We performed numerical tests on the following three different initial data:
\begin{itemize}
  \item[(2a)] Smooth data: $v(x)=\sin(2\pi x)$ which belongs to $H^2(\Omega)\cap H^1_0(\Omega)$.
  \item[(2b)] Nonsmooth data: $v(x)=\chi_{(0,1/2]}$ which lies in the space $\dH {\ep}$ for any $\ep\in [0,1/2)$.
  \item[(2c)] Very weak data: $v(x)=\delta_{1/2}(x)$, a Dirac $\delta_{1/2}(x)$-function
  concentrated at $x=1/2$, which belongs to the space $\dH {- \ep}$ for any $\ep\in (1/2,1]$.
\end{itemize}
In order to check the convergence rate of the semidiscrete scheme, we discretize the fractional
derivatives with a small time step $\tau$ so that the temporal discretization error is negligible.
In view of the possibly singular behavior as $t \rightarrow 0$, we set the time step $\tau$ to
$\tau=t/(5\times 10^4)$, with $t$ being the terminal time. For each example, we measure the error
$e(t)=u(t)-u_h(t)$ by the normalized errors $\| e(t) \|_{L^2(\Omega)}/ \| v \|_{L^2(\Omega)}$ and
$\| \nabla e(t) \|_{L^2(\Omega)}/ \| v \|_{L^2(\Omega)}$. The normalization enables us to observe
the behavior of the error with respect to time in case of nonsmooth initial data.
\subsubsection{Numerical results for example (2a): smooth initial data}
The numerical results show $O(h^2)$ and $O(h)$ convergence rates for the $L^2$- and $H^1$-norms
of the error, respectively, for all three different $\alpha$ values, cf. Fig. \ref{fig:spaceerror:smooth}.
As the value of $\alpha$ increases from $0.25$ to $0.95$, the error at $t=1$ decreases accordingly,
which resembles that for the single-term time-fractional diffusion equation \cite{JinLazarovZhou:2013}.
%
\begin{figure}[htb!]
\center
  \includegraphics[trim = 1cm .2cm 2cm 0.5cm, clip=true,width=10cm]{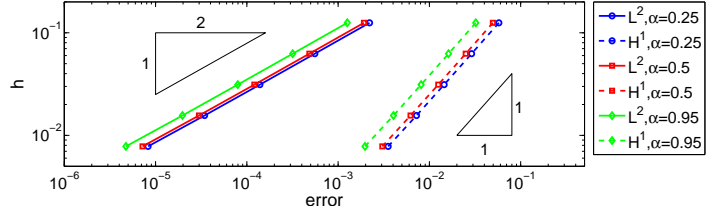}
  \caption{Numerical results for example (2a) at $t=1$ with $\beta=0.2$
and $\al=0.25$, $0.5$, $0.95$, discretized on a uniform mesh $h = 2^{-k}$ and $\tau=2\times10^{-5}$.}
\label{fig:spaceerror:smooth}
\end{figure}

\subsubsection{Numerical results for example (2b): nonsmooth initial data}
For nonsmooth initial data, we are particularly interested in errors for $t$ close to zero,
and thus we also present the errors at $t=0.01$ and $t=0.001$; see Table \ref{tab:nonsmooth1_initial}.
The numerical results fully confirm the theoretically
predicted rates for nonsmooth initial data. Further, in Table \ref{tab:check_singular}
we show the $L^2$-norm of the error for fixed $h = 2^{-6}$ and $t\rightarrow 0$. We observe
that the error deteriorates as $t\rightarrow0$. Upon noting $v\in \dH{1/2-\ep}$, it follows
from Theorem \ref{thm:SG-nonsmooth} that the error grows like $O(t^{-3\al/4})$, which agrees
well with the results in Table \ref{tab:check_singular}.

\begin{table}[htb!]
\caption{Numerical results for the nonsmooth case (2b) with $\al=0.5$ and $\beta=0.2$
at $t=1, 0.01, 0.001$,
discretized on a uniform mesh with $h = 2^{-k}$ and $\tau =t/(5\times10^4)$. }\label{tab:nonsmooth1_initial}
\begin{center}
\begin{tabular}{@{}|c|c|ccccc|c|@{}}
     \hline
      $t$ & $k$ & $3$ & $4$ & $5$ &$6$ & $7$ & rate\\
     \hline
     $t=1$ & $L^2$-norm  & 1.86e-3 &4.64e-4  &1.16e-4  &2.87e-5  &6.88e-6  & $\approx 2.02$ ($2.00$) \\
     & $H^1$-norm           & 4.89e-2 & 2.44e-2 & 1.22e-2 & 6.07e-3 & 2.96e-3 & $\approx 1.01$ ($1.00$)  \\
     \hline
     $t=0.01$ & $L^2$-norm  & 8.04e-3 &2.00e-3  &5.01e-4  &1.24e-4  &2.98e-5 & $\approx 2.03$ ($2.00$)  \\
     & $H^1$-norm          & 2.31e-2 & 1.16e-1 & 5.79e-2 & 2.88e-2 & 1.40e-2  & $\approx 1.01$ ($1.00$)  \\
     \hline
     $t=0.001$ & $L^2$-norm   & 1.65e-2 &4.14e-3  &1.03e-3  &2.56e-4  &6.18e-4 & $\approx 2.01$ ($2.00$)\\
     & $H^1$-norm           & 5.15e-1 & 2.58e-1 & 1.29e-2 & 6.41e-2 & 3.13e-2  & $\approx 1.01$ ($1.00$)\\
     \hline
     \end{tabular}
\end{center}
\end{table}

\begin{table}[htb!]
\caption{$L^2$-error with $\al=0.5$ and $h=2^{-6}$ for $t \to 0$ for
nonsmooth initial data (2b).}\label{tab:check_singular}
\center
     \begin{tabular}{|c|cccccc|c|}
     \hline
     $t$&   1e-3 & 1e-4 & 1e-5 & 1e-6 & 1e-7 & 1e-8  &rate\\
     \hline
     Case(2b)  & 2.56e-4 & 5.39e-4 & 1.15e-3 & 2.91e-3 & 6.77e-3 & 1.55e-2 & $\approx-0.37 (-0.37) $ \\
     \hline
     \end{tabular}
\end{table}

\subsubsection{Numerical results for example (2c): very weak initial data}
The numerical results show a superconvergence with a rate of $O(h^2)$ in the $L^2$-norm and $O(h)$ in the
$H^1$-norm, cf. Fig. \ref{fig:weak}(a). This is attributed to the fact that in one-dimension the solution
with the Dirac $\delta$-function as the initial data is smooth from both sides of the support point and
the finite element spaces $X_h$ have good approximation property. When the singularity point $x=1/2$ is
not aligned with the grid, Fig. \ref{fig:weak}(b) indicates an $O(h^{3/2})$ and $O(h^{1/2})$ convergence
rate for the $L^2$- and $H^1$-norm of the error, respectively, which agrees with our theory.
\begin{figure}[h!]
\center
\subfigure[$x=1/2$ aligns with the grid when $h=2^{-k}$]{
\includegraphics[trim = 1cm .1cm 2cm 0.5cm, clip=true,width=10cm]{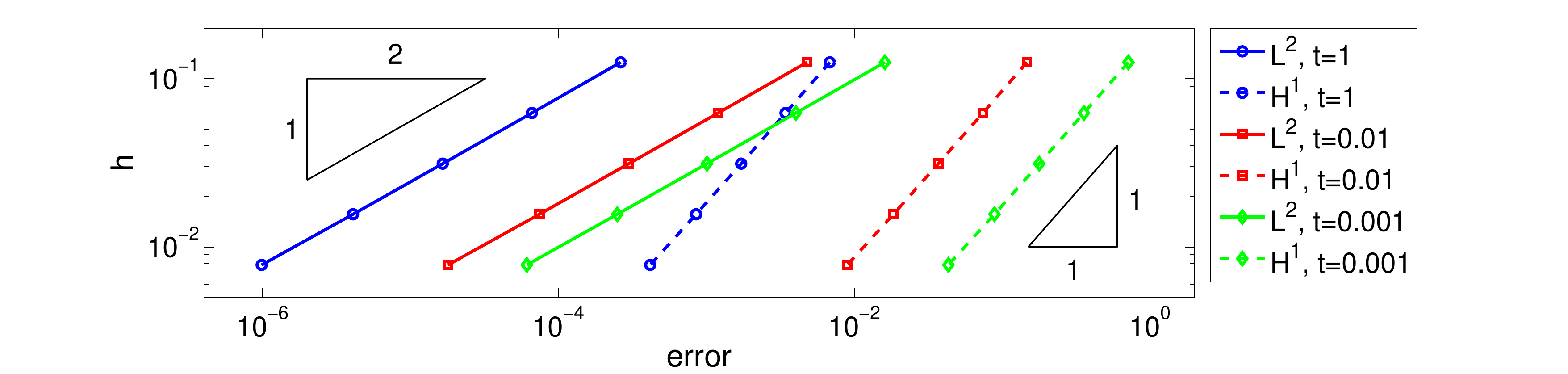}
}
\subfigure[$x=1/2$ does not align with the grid for $h=1/(2^k+1)$]{
\includegraphics[trim = 1cm .1cm 2cm 0.5cm, clip=true,width=10cm]{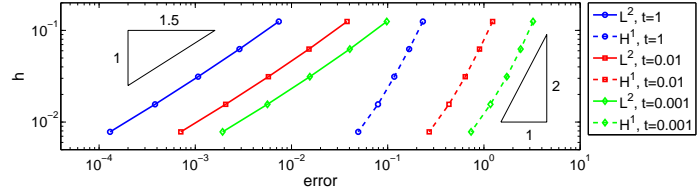}
}
  \caption{Numerical results for very weak initial data: $\al=0.5$, $\beta=0.2$ at $t=0.005$,
  $0.01$, $1.0$, uniform mesh in time with $\tau=t/(5\times 10^{4})$.}
\label{fig:weak}
\end{figure}

\subsection{Inhomogeneous problems}
Now we consider the inhomogeneous problem with $v\equiv0$ on the unit interval $\Omega=(0,1)$ and
test the following two examples:
\begin{enumerate}
 \item[(3a)]  Nonsmooth data: $f(x,t)=(\chi_{[1/2,1]}(t)+1)\chi_{[0,{1/2}]}(x)$.
  The jump at $x=1/2$ leads to
$f(t,\cdot) \notin \dot H^1(\Omega)$; nonetheless,
for any $\epsilon >0$, $f \in L^\infty(0,T;\dH {{1/2}-\epsilon})$.
\item[(3b)] Very weak data: $f(x,t)=(\chi_{[1/2,1]}(t)+1)\delta_{1/2}(x)$
where $f$ involves a Dirac $\delta_{1/2}(x)$-function
concentrated at $x=0.5$.
\end{enumerate}

\subsubsection{Numerical results for example (3a)}
Since the errors are bounded independently of the time, cf. Theorem \ref{thm:gal:linf},
we only present the errors in $L^\infty$ in time, i.e., $\|e(t)\|_{L^2(\Omega)}$ and $\| \nabla e(t)
\|_{L^2(\Omega)}$. In Table \ref{tab:nonsmooth1_right}, we present the $L^2$- and $H^1$-error
at $t=1$, $0.01$, and $0.001$. The numerical results agree well with our theoretical predictions,
i.e., $O(h^2)$ and $O(h)$ convergence rates for the $L^2$- and $H^1$-norms of the error, respectively.

\begin{table}[htb!]
\caption{Numerical results for example (3a) with $\al=0.5$ and $\beta=0.2$
at $t=1, 0.01, 0.001$,
discretized on a uniform mesh $h = 2^{-k}$ and $\tau=t/(5\times 10^{4})$. }\label{tab:nonsmooth1_right}
\begin{center}
\begin{tabular}{@{}|c|c|ccccc|c|@{}}
     \hline
      $t$ & $k$ & $3$ & $4$ & $5$ &$6$ & $7$ & rate\\
     \hline
     $t=1$ & $L^2$-norm  & 1.76e-3 &4.40e-4  &1.10e-4  &2.71e-5  &6.53e-6  & $\approx 2.01$ ($2.00$) \\
     & $H^1$-norm           & 4.72e-2 & 2.36e-2 & 1.18e-2 & 5.86e-3 & 2.86e-3 & $\approx 1.01$ ($1.00$)  \\
     \hline
     $t=0.01$ & $L^2$-norm  & 6.34e-4 &1.59e-4  &3.96e-5  &9.82e-6  &2.38e-6 & $\approx 2.01$ ($2.00$)  \\
     & $H^1$-norm          & 1.89e-2 & 9.46e-3 & 4.72e-3 & 2.35e-3 & 1.15e-3  & $\approx 1.01$ ($1.00$)  \\
     \hline
     $t=0.001$ & $L^2$-norm   & 4.55e-4 &1.15e-4  &2.88e-5  &1.15e-6  &1.73e-6 & $\approx 2.02$ ($2.00$)\\
     & $H^1$-norm           & 1.45e-2 & 7.31e-3 & 3.66e-3 & 1.82e-3 & 8.88e-4  & $\approx 1.01$ ($1.00$)\\
     \hline
     \end{tabular}
\end{center}
\end{table}
\subsubsection{Numerical results for example (3b)}
In Table \ref{tab:weak_right_G} we show convergence rates at three different times,
i.e., $t=1$, $0.01$, and $0.001$. Here the mesh size $h$ is chosen to be $h=1/(2^k+1)$, and
thus the support of the Dirac $\delta$-function does not align with the grid. The results indicate
an $O(h^{1/2})$ and $O(h^{3/2})$ convergence rate for the $H^1$- and $L^2$-norm of the error,
respectively, which agrees well with the theoretical prediction. If the Dirac $\delta$-function
is supported at a grid point, both $L^2$- and $H^1$-norms of the error exhibit a superconvergence
of one half order, cf. Table \ref{tab:weak_right_d}. This, however, theoretically remains to be established.

\begin{table}[htb!]
\caption{Numerical results for example (3b) with $\al=0.5$ and $\beta=0.2$
at $t=0.1, 0.01, 0.001$,
discretized on a uniform mesh $h = 1/(2^{k}+1)$ and $\tau=t/(5\times 10^{4})$. }\label{tab:weak_right_d}
\begin{center}
\begin{tabular}{@{}|c|c|ccccc|c|@{}}
     \hline
      $t$ & $k$ & $3$ & $4$ & $5$ &$6$ & $7$ & rate\\
     \hline
     $t=0.1$ & $L^2$-norm  & 1.02e-2 &4.01e-3  &1.49e-3  &5.35e-4  &1.82e-4  & $\approx 1.49$ ($1.50$) \\
     & $H^1$-norm           & 3.24e-1 & 2.35e-1 & 1.65e-1 & 1.11e-1 & 6.94e-2 & $\approx 0.50$ ($0.50$)  \\
     \hline
     $t=0.01$ & $L^2$-norm  & 4.66e-3 &1.91e-3  &7.29e-4  &2.64e-4  &9.02e-5 & $\approx 1.45$ ($1.50$)  \\
     & $H^1$-norm          & 1.54e-1 & 1.14e-1 & 8.16e-2 & 5.54e-2 & 3.47e-2  & $\approx 0.55$ ($0.50$)  \\
     \hline
     $t=0.001$ & $L^2$-norm   & 4.30e-3 &1.83e-3  &7.12e-4  &2.61e-4  &8.97e-5 & $\approx 1.45$ ($1.50$)\\
     & $H^1$-norm           & 1.47e-1 & 1.11e-1 & 8.05e-2 & 5.50e-2 & 3.45e-2  & $\approx 0.55$ ($0.50$)\\
     \hline
     \end{tabular}
\end{center}
\end{table}
\vspace{-.2cm}
\begin{table}[htb!]
\caption{Numerical results for example (3b) with $\al=0.5$ and $\beta=0.2$
at $t=1, 0.01, 0.001$,
discretized on a uniform mesh with $h= 2^{-k}$ and $\tau=t/(5\times 10^{4})$. }\label{tab:weak_right_G}
\begin{center}
\begin{tabular}{@{}|c|c|ccccc|c|@{}}
     \hline
      $t$ & $k$ & $3$ & $4$ & $5$ &$6$ & $7$ & rate\\
     \hline
     $t=1$ & $L^2$-norm  & 5.35e-4 &1.34e-4  &3.35e-5  &8.31e-6  &2.01e-6  & $\approx 2.01$ ($1.50$) \\
     & $H^1$-norm           & 1.49e-2 & 7.48e-3 & 3.74e-3 & 1.86e-3 & 9.07e-4 & $\approx 1.01$ ($0.50$)  \\
     \hline
     $t=0.01$ & $L^2$-norm  & 6.67e-4 &1.67e-4  &4.17e-5  &1.04e-5  &2.52e-6 & $\approx 2.03$ ($1.50$)  \\
     & $H^1$-norm          & 2.56e-2 & 1.29e-2 & 6.44e-3 & 3.20e-3 & 1.56e-3  & $\approx 1.02$ ($0.50$)  \\
     \hline
     $t=0.001$ & $L^2$-norm   & 8.19e-4 &2.08e-4  &5.22e-5  &1.30e-5  &3.19e-6 & $\approx 2.02$ ($1.50$)\\
     & $H^1$-norm           & 3.96e-2 & 2.00e-2 & 1.00e-3 & 4.98e-3 & 2.45e-3  & $\approx 1.01$ ($0.50$)\\
     \hline
     \end{tabular}
\end{center}
\end{table}

\subsection{Examples in two-dimension}
In this part, we present three two-dimensional examples on the unit square $\Omega=(0,1)^2$.
\begin{enumerate}
 \item[(4a)]  Nonsmooth initial data: $v=\chi_{(0,1/2)\times(0,1)}$ and $f\equiv0$.
 \item[(4b)] Very weak initial data: $v=\delta_\Gamma$ with $\Gamma$ being the boundary of the square
  $[1/4,3/4]^2$ and $\langle \delta_\Gamma,\phi\rangle=\int_\Gamma\phi(s)\,ds$.
  By H\"{o}lder's inequality and the continuity of the trace operator from $\dot H^{{1/2}+\epsilon}(\Omega)$
  to $L^2(\Gamma)$ \cite{AdamsFournier:2003}, we deduce $\delta_\Gamma \in H^{-1/2-\epsilon}(\Omega)$.
 \item[(4c)]  Nonsmooth right hand side: $f(x,t)=(\chi_{[1/20,1/10]}(t)+1)\chi_{(0,1/2)\times(0,1)}(x)$ and $v\equiv0$.
\end{enumerate}

To discretize the problem, we divide each direction into $N=2^k$ equally spaced
subintervals, with a mesh size $h=1/N$ so that the domain $[0,1]^2$ is divided into
$N^2$ small squares. We get a symmetric mesh by connecting the diagonal of each small square.

The numerical results for example (4a) are shown in Table \ref{tab:nonsmooth2D}, which agree well
with Theorem \ref{thm:SG-nonsmooth}, with a rate $O(h^2)$ and $O(h)$,
respectively, for the $L^2$- and $H^1$-norm of the error. Interestingly, for example (4b), both
the $L^2$-norm and $H^1$-norm of the error exhibit super-convergence, cf. Table \ref{tab:weak2D_G}.
The numerical results for example (4c) confirm the theoretical results; see Table \ref{tab:2Dnonsmooth_right}.
The solution profiles for examples (4b) and (4c) at $t=0.1$ are shown in Fig. \ref{fig:2D_solution},
from which the nonsmooth region of the solution can be clearly observed.

\begin{table}[hbt!]
\caption{Numerical results for (4a) with $\al=0.5$ and $\beta=0.2$
at $t=0.1, 0.01, 0.001$, discretized on a uniform mesh, $h = 2^{-k}$ and $\tau=t/10^{4}$. }\label{tab:nonsmooth2D}
\begin{center}
\begin{tabular}{@{}|c|c|ccccc|c|@{}}
     \hline
      $t$ & $k$ & $3$ & $4$ & $5$ &$6$ & $7$ & rate\\
     \hline
     $t=0.1$ & $L^2$-norm  & 5.25e-3 &1.35e-3  &3.38e-4  &8.24e-5  &1.98e-5  & $\approx 2.06$ ($2.00$) \\
     & $H^1$-norm           & 9.10e-2 & 4.53e-2 & 2.25e-2 & 1.09e-2 & 4.99e-3 & $\approx 1.04$ ($1.00$)  \\
     \hline
     $t=0.01$ & $L^2$-norm  & 1.25e-2 &3.23e-3  &8.09e-4  &1.97e-4  &4.65e-5 & $\approx 2.05$ ($2.00$)  \\
     & $H^1$-norm          & 2.18e-1 & 1.08e-1 & 5.35e-2 & 2.62e-2 & 1.27e-2  & $\approx 1.05$ ($1.00$)  \\
     \hline
     $t=0.001$ & $L^2$-norm   & 3.02e-2 &7.84e-3  &1.97e-3  &4.81e-4  &1.16e-4& $\approx 2.03$ ($2.00$)\\
     & $H^1$-norm           & 5.30e-1 & 2.64e-1 & 1.31e-1 & 6.38e-2 & 3.14e-2  & $\approx 1.04$ ($1.00$)\\
     \hline
     \end{tabular}
\end{center}
\end{table}

\vspace{-.6cm}
\begin{figure}[htb!]
  \includegraphics[trim = 1cm .1cm 2cm 0.0cm, clip=true,width=10cm]{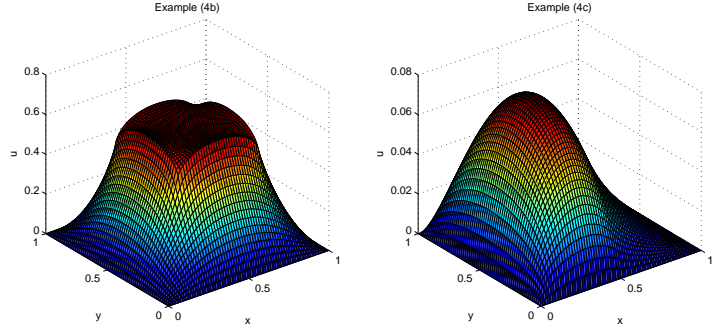}
  \caption{Numerical solutions of examples (4b) and (4c) with $h = 2^{-6}$,
  $\alpha=0.5$, $\beta=0.2$ at $t=0.1$}\label{fig:2D_solution}
\end{figure}

\begin{table}[hbt!]
\caption{Numerical results for example (4b) with $\al=0.5$ and $\beta=0.2$
at $t=0.1, 0.01, 0.001$ for a uniform mesh with
$h = 2^{-k}$ and $\tau=t/10^{4}$. }\label{tab:weak2D_G}
\begin{center}
\begin{tabular}{@{}|c|c|ccccc|c|@{}}
     \hline
      $t$ & $k$ & $3$ & $4$ & $5$ &$6$ & $7$ & rate\\
     \hline
     $t=0.1$ & $L^2$-norm  & 1.18e-2 &3.18e-3  &8.41e-4  &2.18e-4  &5.41e-5  & $\approx 1.92$ ($1.50$) \\
     & $H^1$-norm           & 2.25e-1 & 1.13e-1 & 6.60e-2 & 3.40e-2 & 1.66e-2 & $\approx 0.92$ ($0.50$)  \\
     \hline
     $t=0.01$ & $L^2$-norm  & 2.82e-2 &7.62e-3  &2.28e-3  &5.26e-4  &1.25e-4 & $\approx 1.95$ ($2.00$)  \\
     & $H^1$-norm          & 5.66e-1 & 3.09e-1 & 1.65e-1 & 8.52e-2 & 4.19e-2  & $\approx 0.94$ ($1.00$)  \\
     \hline
     $t=0.001$ & $L^2$-norm   & 6.65e-2 &1.83e-3  &4.98e-3  &1.33e-3  &3.30e-4& $\approx 1.91$ ($2.00$)\\
     & $H^1$-norm           & 1.66e0 & 8.93e-1 & 4.75e-1 & 2.43e-1 & 1.21e-1  & $\approx 0.95$ ($1.00$)\\
     \hline
     \end{tabular}
\end{center}
\end{table}

\begin{table}[hbt!]
\caption{Numerical results for example (4c) with $\al=0.5$ and $\beta=0.2$ at $t=0.1, 0.01, 0.001$
for a uniform mesh with $h= 2^{-k}$ and $\tau=t/10^{4}$. }\label{tab:2Dnonsmooth_right}
\begin{center}
\begin{tabular}{@{}|c|c|ccccc|c|@{}}
     \hline
      $t$ & $k$ & $3$ & $4$ & $5$ &$6$ & $7$ & rate\\
     \hline
     $t=0.1$ & $L^2$-norm  & 2.28e-3 &5.86e-4  &1.47e-4  &3.58e-5 & 7.91e-6  & $\approx 2.07$ ($2.00$) \\
     & $H^1$-norm           & 3.97e-2 & 1.97e-2 & 9.77e-3 & 4.76e-3 & 2.13e-3 & $\approx 1.06$ ($1.00$)  \\
     \hline
     $t=0.01$ & $L^2$-norm  & 1.06e-3 &2.73e-4  &6.86e-5  &1.67e-6  &3.70e-6 & $\approx 2.06$ ($2.00$)  \\
     & $H^1$-norm          & 1.85e-2 & 9.18e-3 & 4.56e-3 & 2.22e-3 & 9.94e-3  & $\approx 1.06$ ($1.00$)  \\
     \hline
     $t=0.001$ & $L^2$-norm   & 8.66e-4 &2.28e-4  &5.75e-5  &1.40e-6  &3.11e-6 & $\approx 2.04$ ($2.00$)\\
     & $H^1$-norm           & 1.56e-2 & 7.82e-3 & 3.88e-3 & 1.90e-3 & 8.47e-4  & $\approx 1.05$ ($1.00$)\\
     \hline
     \end{tabular}
\end{center}
\end{table}

\section{Concluding remarks}
In this work, we have developed a simple numerical scheme based on the Galerkin finite element method
for a multi-term time fractional diffusion equation which involves multiple Caputo fractional
derivatives in time. A complete error analysis of the space semidiscrete Galerkin scheme is provided. The
theory covers the practically very important case of nonsmooth initial data and right hand side.
The analysis relies essentially on some new regularity results of the multi-term time fractional
diffusion equation. Further, we have developed a fully discrete scheme based on a finite difference
discretization of the Caputo fractional derivatives. The stability and error estimate of the fully
discrete scheme were established, provided that the solution is smooth. The extensive
numerical experiments in one- and two-dimension fully confirmed our convergence analysis: the empirical
convergence rates agree well with the theoretical predictions for both smooth and nonsmooth data.

\section*{Acknowledgements}
The research of B. Jin has been supported by US NSF Grant DMS-1319052, R. Lazarov was supported in parts
by US NSF Grant DMS-1016525 and also by Award No. KUS-C1-016-04, made by King Abdullah University of Science and Technology (KAUST), and 
Y. Liu was supported by the Program for Leading Graduate Schools, MEXT, Japan.

\appendix
\section{Proof of Lemma \ref{lem:MMLcm}}\label{app:MMLcm}
\begin{proof}
First, we define an auxiliary function $v_j(t)$ by
\begin{equation*}
  \hat v_j(z)=\mathcal{L}(v_j)
  = \frac{z^{\al-1}+\sum_{k=1}^m b_k z^{\al_k-1}}{z^\al+\sum_{k=1}^m b_k z^{\al_k}+\la_j}.
\end{equation*}
Now by the property of the Laplace transform $f(0^+)=\lim_{z\to \infty} z \widehat{f}(z)$,
we obtain $v_j(0+)=1$.
The function $\bar E_{\vecal}^j(t)$ is the inverse Laplace integral of
$\widehat {\bar E_{\vecal}^j} =(z^\al+\sum_{k=1}^m b_k z^{\al_k}+\la_j)^{-1}$, i.e.
\begin{equation}\label{L1}
  \bar E_{\vecal}^j(t)=\frac{1}{2\pi\mathrm{i}} \int_{Br}e^{zt}\frac{1}{{z^\al+\sum_{k=1}^m b_k z^{\al_k}+\la_j}}\,dz,
\end{equation}
where $Br=\{z;\ \text{Re}~~z=\sigma,\ \sigma>0\}$ is the Bromwich path. The function
$\widehat{\bar E_{\vecal}^j}(z)$ has a branch point $0$, so we cut off the negative
part of the real axis. Note that the function $z^\al+\sum_{k=1}^m b_k z^{\al_k}+\la_j$
has no zero in the main sheet of the Riemann surface including its boundaries on the cut.
Indeed, if $z=\varrho e^{i\theta}$, with $\rho>0$, $\theta\in(-\pi,\pi)$, then
\begin{equation*}
   \Im\left\{z^\al+\sum_{k=1}^m b_k z^{\al_k}+\la_j\right\}=\varrho^\al\sin\al\theta+\sum_{k=1}^m b_k \varrho^{\al_k}\sin\al_k\theta\neq 0,\ \ \forall \theta\neq 0,
\end{equation*}
since $\sin \alpha \theta$ and $\sin\alpha_k\theta$ have the same sign for any $\theta\in(-\pi,\pi)$
and $b_k>0$. Hence, $ \bar E_{\vecal}^j(t)$ can be found by bending the Bromwich path
into the Hankel path $Ha(\ep)$, which starts from $-\infty$ along the lower side
of the negative real axis, encircles the disc $|s|=\ep$ counterclockwise and ends
at $-\infty$ along the upper side of the negative real axis.
Then by taking $\ep \to 0$, we obtain
\begin{equation*} 
  \bar E_{\vecal}^j(t)= \int_0^\infty e^{-rz}K_n(r)\,dr,
\end{equation*}
where
\begin{equation*} 
  K_n(r)=-\frac{1}{\pi}\Im\left\{\left.\frac{1}
  {z^\al+\sum_{k=1}^m b_k z^{\al_k}+\la_j}\right|_{z=r e^{i\pi}}\right\}.
\end{equation*}
It is easy to check
\begin{equation*}
 K_n(r)=\frac{1}{\pi} \frac{r^{\al}\sin\al\pi+\sum_{k=1}^m b_kr^{\al_k}\sin\al_k\pi}
 {(r^{\al}\cos\al\pi+\sum_{k=1}^m b_kr^{\al_k}\cos\al_k\pi+\la_j)^2
 +(r^{\al}\sin\al\pi+\sum_{k=1}^m b_kr^{\al_k}\sin\al_k\pi)^2}
\end{equation*}
which is greater than zero for all $r>0$. Therefore, $\bar E_{\vecal}^j(t)$
is completely monotone. A similar argument shows that $v_j(t)$ is also
completely monotone. Consequently,
\begin{equation*}
\int_0^T |\bar E_{\vecal}^j(t)|\,dt=\int_0^T \bar E_{\vecal}^j(t)\,dt=
-\frac{1}{\la_j}\int_0^T v'_j(t)\,dt=\frac{1}{\la_n}(1-v_j(T))<\frac{1}{\la_n},
\end{equation*}
which concludes the proof of the lemma.
\end{proof}

\bibliographystyle{abbrv}
\bibliography{frac}

\begin{thebibliography}{10}

\bibitem{AdamsGelhar:1992}
E.~E. Adams and L.~W. Gelhar.
\newblock Field study of dispersion in a heterogeneous aquifer: 2. spatial
  moments analysis.
\newblock {\em Water Res. Research}, 28(12):3293--3307, 1992.

\bibitem{AdamsFournier:2003}
R.~Adams and J.~Fournier.
\newblock {\em Sobolev {S}paces}.
\newblock Elsevier/Academic Press, Amsterdam, 2003.

\bibitem{BrunnerLingYamamoto:2010}
H.~Brunner, L.~Ling, and M.~Yamamoto.
\newblock Numerical simulations of {2D} fractional subdiffusion problems.
\newblock {\em J. Comput. Phys.}, 229(18):6613–--6622, 2010.

\bibitem{CasasClasonKunisch:2013}
E.~Casas, C.~Clason, and K.~Kunisch.
\newblock Parabolic control problems in measure spaces with sparse solutions.
\newblock {\em SIAM J. Control Optim.}, 51(1):28--63, 2013.

\bibitem{CasasZuazua:2013}
E.~Casas and E.~Zuazua.
\newblock Spike controls for elliptic and parabolic {PDE}s.
\newblock {\em Systems Control Lett.}, 62(4):311--318, 2013.

\bibitem{ElSayedElKallaZiada:2010}
A.~M.~A. El-Sayed, I.~L. El-Kalla, and E.~A.~A. Ziada.
\newblock Analytical and numerical solutions of multi-term nonlinear fractional
  orders differential equations.
\newblock {\em Appl. Numer. Math.}, 60(8):788--797, 2010.

\bibitem{FuChenYang:2013}
Z.-J. Fu, W.~Chen, and H.-T. Yang.
\newblock Boundary particle method for {L}aplace transformed time fractional
  diffusion equations.
\newblock {\em J. Comput. Phys.}, 235:52--66, 2013.

\bibitem{HadidLuchko:1996}
S.~B. Hadid and Y.~F. Luchko.
\newblock An operational method for solving fractional differential equations
  of an arbitrary real order.
\newblock {\em Panamer. Math. J.}, 6(1):57--73, 1996.

\bibitem{JiangLiuTurnerBurrage:2012b}
H.~Jiang, F.~Liu, I.~Turner, and K.~Burrage.
\newblock Analytical solutions for the multi-term time-space {C}aputo-{R}iesz
  fractional advection-diffusion equations on a finite domain.
\newblock {\em J. Math. Anal. Appl.}, 389(2):1117--1127, 2012.

\bibitem{JinLazarovPasciakZhou:2013a}
B.~Jin, R.~Lazarov, J.~Pasciak, and Z.~Zhou.
\newblock Error analysis of semidiscrete finite element methods for
  inhomogeneous time-fractional diffusion.
\newblock preprint, arXiv:1307.1068, 2013.

\bibitem{JinLazarovPasciakZhou:2013}
B.~Jin, R.~Lazarov, J.~Pasciak, and Z.~Zhou.
\newblock Galerkin fem for fractional order parabolic equations with initial
  data in {$H^{-s},~0\le s \le 1$}.
\newblock Proc. 5th Conf. Numer. Anal. Appl., Springer, 24--37, 2013.

\bibitem{JinLazarovZhou:2013}
B.~Jin, R.~Lazarov, and Z.~Zhou.
\newblock Error estimates for a semidiscrete finite element method for
  fractional order parabolic equations.
\newblock {\em SIAM J. Numer. Anal.}, 51(1):445--466, 2013.

\bibitem{JinLu:2012}
B.~Jin and X.~Lu.
\newblock Numerical identification of a {R}obin coefficient in parabolic
  problems.
\newblock {\em Math. Comp.}, 81:1369--1398, 2012.

\bibitem{JinRundell:2012}
B.~Jin and W.~Rundell.
\newblock An inverse problem for a one-dimensional time-fractional diffusion
  problem.
\newblock {\em Inverse Problems}, 28(7):075010, 19, 2012.

\bibitem{Katsikadelis:2009}
J.~Katsikadelis.
\newblock Numerical solution of multi-term fractional differential equations.
\newblock {\em ZAMM Z. Angew. Math. Mech.}, 89(7):593--608, 2009.

\bibitem{KellyMcGoughMeerschaert:2008}
J.~F. Kelly, R.~J. {McGough}, and M.~M. Meerschaert.
\newblock Analytical time-domain {G}reen's functions for power-law media.
\newblock {\em J. Acoust. Soc. Am.}, 124(5):2861--2872, 2008.

\bibitem{KilbasSrivastavaTrujillo:2006}
A.~Kilbas, H.~Srivastava, and J.~Trujillo.
\newblock {\em Theory and {A}pplications of {F}ractional {D}ifferential
  {E}quations}.
\newblock Elsevier, Amsterdam, 2006.

\bibitem{LiXu:2009}
X.~Li and C.~Xu.
\newblock A space-time spectral method for the time fractional diffusion
  equation.
\newblock {\em SIAM J. Numer. Anal.}, 47(3):2108--2131, 2009.

\bibitem{LiLiuYamamoto:2013}
Z.~Li, Y.~Liu, and M.~Yamamoto.
\newblock Initial-boundary value problems for multi-term time-fractional
  diffusion equations with positive constant coefficients.
\newblock preprint, arXiv:1312.2112, 2013.

\bibitem{LiYamamoto:2013}
Z.~Li and M.~Yamamoto.
\newblock Initial-boundary value problems for linear diffusion equations with
  multiple time-fractional derivatives.
\newblock preprint, arXiv:1306.2778, 2013.

\bibitem{LinXu:2007}
Y.~Lin and C.~Xu.
\newblock Finite difference/spectral approximations for the time-fractional
  diffusion equation.
\newblock {\em J. Comput. Phys.}, 225(2):1533--1552, 2007.

\bibitem{LiuMeerschaert:2013}
F.~Liu, M.~M. Meerschaert, R.~J. {McGough}, P.~Zhuang, and X.~Liu.
\newblock Numerical methods for solving the multi-term time-fractional
  wave-diffusion equation.
\newblock {\em Frac. Cal. Appl. Anal.}, 16(1):9--25, 2013.

\bibitem{Luchko:2011}
Y.~Luchko.
\newblock Initial-boundary-value problems for the generalized multi-term
  time-fractioal diffusion equations.
\newblock {\em J. Math. Anal. Appl.}, 374(2):538--548, 2011.

\bibitem{LuchkoGorenflo:1999}
Y.~Luchko and R.~Gorenflo.
\newblock An operational method for solving fractional differential equations
  with the {C}aputo derivatives.
\newblock {\em Acta Math. Vietnam.}, 24(2):207--233, 1999.

\bibitem{McLeanThomee:2010}
W.~{McLean} and V.~Thom{\'e}e.
\newblock Maximum-norm error analysis of a numerical solution via {L}aplace
  transformation and quadrature of a fractional-order evolution equation.
\newblock {\em IMA J. Numer. Anal.}, 30(1):208--230, 2010.

\bibitem{MetzlerKlafterSokolov:1998}
R.~Metzler, J.~Klafter, and I.~M. Sokolov.
\newblock Anomalous transport in external fields: Continuous time random walks
  and fractional diffusion equations extended.
\newblock {\em Phys. Rev. E}, 58(2):1621--1633, 1998.

\bibitem{Mustapha:2011}
K.~Mustapha.
\newblock An implicit finite-difference time-stepping method for a
  sub-diffusion equation, with spatial discretization by finite elements.
\newblock {\em IMA J. Numer. Anal.}, 31(2):719--739, 2011.

\bibitem{MustaphaMcLean:2013}
K.~Mustapha and W.~McLean.
\newblock Superconvergence of a discontinuous {G}alerkin method for fractional
  diffusion and wave equations.
\newblock {\em SIAM J. Numer. Anal.}, 51(1):491--515, 2013.

\bibitem{Nigmatulin:1986}
R.~Nigmatulin.
\newblock The realization of the generalized transfer equation in a medium with
  fractal geometry.
\newblock {\em Phys. Stat. Sol. B}, 133:425--430, 1986.

\bibitem{SakamotoYamamoto:2011}
K.~Sakamoto and M.~Yamamoto.
\newblock Initial value/boundary value problems for fractional diffusion-wave
  equations and applications to some inverse problems.
\newblock {\em J.Math.Anal.Appl.}, 382(1):426--447, 2011.

\bibitem{SchumerBensonMeerschaertBaeumer:2003}
R.~Schumer, D.~A. Benson, M.~M. Meerschaert, and B.~Baeumer.
\newblock Fractal mobile/immobile solute transport.
\newblock {\em Water Res. Research.}, 39(10):1296, 13 pp., 2003.

\bibitem{Thomee:2006}
V.~Thom{\'e}e.
\newblock {\em Galerkin {F}inite {E}lement {M}ethods for {P}arabolic
  {P}roblems}, volume~25 of {\em Springer Series in Computational Mathematics}.
\newblock Springer-Verlag, Berlin, 2006.

\bibitem{XieZou:2005}
J.~Xie and J.~Zou.
\newblock Numerical reconstruction of heat fluxes.
\newblock {\em SIAM J. Numer. Anal.}, 43(4):1504--1535, 2005.

\bibitem{ZayernouriKardiadakis:2014}
M.~Zayernouri and G.~E. Karniadakis.
\newblock Exponentially accurate spectral and spectral element methods for
  fractional {ODE}s.
\newblock {\em J. Comput. Phys.}, 257, Part A:460--–480, 2014.

\bibitem{ZhangSunWu:2011}
Y.-N. Zhang, Z.-Z. Sun, and H.-W. Wu.
\newblock Error estimates of {C}rank-{N}icolson-type difference schemes for the
  subdiffusion equation.
\newblock {\em SIAM J. Numer. Anal.}, 49(6):2302--2322, 2011.

\bibitem{ZhaoXiaoXu:2013}
J.~Zhao, J.~Xiao, and Y.~Xu.
\newblock Stability and convergence of an effective finite element method for
  multiterm fractional partial differential equations.
\newblock {\em Abstr. Appl. Anal.}, pages Art. ID 857205, 10, 2013.

\end{thebibliography}
\end{document}